\documentclass{amsart}

\usepackage{amsfonts,amssymb,verbatim,amsmath,amsthm,latexsym,textcomp,amscd}
\usepackage{latexsym,amsfonts,amssymb,epsfig,verbatim}
\usepackage{amsmath,amsthm,amssymb,latexsym,graphics,textcomp}
\usepackage{graphicx}
\usepackage{color}
\usepackage{url}
\usepackage{psfrag}

\begin{document}

\newtheorem{theorem}{Theorem}[section]
\newtheorem{prop}[theorem]{Proposition}
\newtheorem{lemma}[theorem]{Lemma}
\newtheorem{cor}[theorem]{Corollary}
\newtheorem{definition}[theorem]{Definition}
\newtheorem{defn}[theorem]{Definition}
\newtheorem{conj}[theorem]{Conjecture}
\newtheorem{claim}[theorem]{Claim}
\newtheorem{defth}[theorem]{Definition-Theorem}
\newtheorem{obs}[theorem]{Observation}
\newtheorem{rmark}[theorem]{Remark}
\newtheorem{qn}[theorem]{Question}
\newtheorem{theo}[theorem]{Theorem}
\newtheorem{thmbis}{Theorem}
\newtheorem{dfn}[theorem]{Definition} 
\newtheorem{defi}[theorem]{Definition} 
\newtheorem{coro}[theorem]{Corollary}
\newtheorem{corbis}{Corollary}
\newtheorem{propbis}{Proposition} 
\newtheorem*{prop*}{Proposition} 
\newtheorem{lem}[theorem]{Lemma} 
\newtheorem{lembis}{Lemma} 
\newtheorem{claimbis}{Claim} 
\newtheorem{fact}[theorem]{Fact} 
\newtheorem{factbis}{Fact} 
\newtheorem{qst}[theorem]{Question} 
\newtheorem{qstbis}{Question} 
\newtheorem{pb}[theorem]{Problem} 
\newtheorem{pbbis}{Problem} 
 \newtheorem{question}[theorem]{Question}
\newtheorem{rem}[theorem]{Remark}
\newtheorem{remark}[theorem]{Remark}
\newtheorem{example}[theorem]{Example}
\newtheorem{eg}[theorem]{Example}
\newtheorem{notation}[theorem]{Notation}
\newenvironment{preuve}[1][Preuve]{\begin{proof}[#1]}{\end{proof}}

\newcommand{\hhat}{\widehat}
\newcommand{\boundary}{\partial}
\newcommand{\C}{{\mathbb C}}
\newcommand{\bD}{{\mathbb D}}
\newcommand{\bS}{{\mathbb S}}
\newcommand{\integers}{{\mathbb Z}}
\newcommand{\natls}{{\mathbb N}}
\newcommand{\bbN}{{\mathbb N}}
\newcommand{\bH}{{\mathbb H}}
\newcommand{\ratls}{{\mathbb Q}}
\newcommand{\reals}{{\mathbb R}}
\newcommand{\bbR}{{\mathbb R}}
\newcommand{\lhp}{{\mathbb L}}
\newcommand{\tube}{{\mathbb T}}
\newcommand{\cusp}{{\mathbb P}}
\newcommand\AAA{{\mathcal A}}
\newcommand\BB{{\mathcal B}}
\newcommand\CC{{\mathcal C}}
\newcommand\DD{{\mathcal D}}
\newcommand\EE{{\mathcal E}}
\newcommand\FF{{\mathcal F}}
\newcommand\GG{{\mathcal G}}
\newcommand\HH{{\mathcal H}}
\newcommand\II{{\mathcal I}}
\newcommand\JJ{{\mathcal J}}
\newcommand\KK{{\mathcal K}}
\newcommand\LL{{\mathcal L}}
\newcommand\MM{{\mathcal M}}
\newcommand\nN{{\mathcal N}}
\newcommand\OO{{\mathcal O}}
\newcommand\PP{{\mathcal P}}
\newcommand\QQ{{\mathcal Q}}
\newcommand\RR{{\mathcal R}}
\newcommand\SSS{{\mathcal S}}
\newcommand\TT{{\mathcal T}}
\newcommand\UU{{\mathcal U}}
\newcommand\VV{{\mathcal V}}
\newcommand\WW{{\mathcal W}}
\newcommand\XX{{\mathcal X}}
\newcommand\YY{{\mathcal Y}}
\newcommand\ZZ{{\mathcal Z}}
\newcommand\CH{{\CC\Hyp}}
\newcommand\Ga{{\Gamma}}
\newcommand\MF{{\MM\FF}}
\newcommand\PMF{{\PP\kern-2pt\MM\FF}}
\newcommand\ML{{\MM\LL}}
\newcommand\PML{{\PP\kern-2pt\MM\LL}}
\newcommand\GL{{\GG\LL}}
\newcommand\Pol{{\mathcal P}}
\newcommand\half{{\textstyle{\frac12}}}
\newcommand\Half{{\frac12}}
\newcommand\Mod{\operatorname{Mod}}
\newcommand\Area{\operatorname{Area}}
\newcommand\ep{\epsilon}
\newcommand\Hypat{\widehat}
\newcommand\Proj{{\mathbf P}}
\newcommand\U{{\mathbf U}}
 \newcommand\Hyp{{\mathbf H}}
\newcommand\D{{\mathbf D}}
\newcommand\Z{{\mathbb Z}}
\newcommand\R{{\mathbb R}}
\newcommand\Q{{\mathbb Q}}
\newcommand\E{{\mathbb E}}
\newcommand\til{\widetilde}
\newcommand\length{\operatorname{length}}
\newcommand\tr{\operatorname{tr}}
\newcommand\gesim{\succ}
\newcommand\lesim{\prec}
\newcommand\simle{\lesim}
\newcommand\simge{\gesim}
\newcommand{\simmult}{\asymp}
\newcommand{\simadd}{\mathrel{\overset{\text{\tiny $+$}}{\sim}}}
\newcommand{\ssm}{\setminus}
\newcommand{\pair}[1]{\langle #1\rangle}
\newcommand{\T}{{\mathbf T}}
\newcommand{\inj}{\operatorname{inj}}
\newcommand{\collar}{\operatorname{\mathbf{collar}}}
\newcommand{\bcollar}{\operatorname{\overline{\mathbf{collar}}}}
\newcommand{\I}{{\mathbf I}}
\newcommand{\eps}{\epsilon}

\newcommand{\bbar}{\overline}
\newcommand{\UML}{\operatorname{\UU\MM\LL}}
\newcommand{\EL}{\mathcal{EL}}
\newcommand\MT{{\mathbb T}}
\newcommand\Teich{{\mathcal T}}

\makeatletter
\@tfor\next:=abcdefghijklmnopqrstuvwxyzABCDEFGHIJKLMNOPQRSTUVWXYZ\do{%
  \def\command@factory#1{%
    \expandafter\def\csname cal#1\endcsname{\mathcal{#1}}
    \expandafter\def\csname frak#1\endcsname{\mathfrak{#1}}
    \expandafter\def\csname scr#1\endcsname{\mathscr{#1}}
    \expandafter\def\csname bb#1\endcsname{\mathbb{#1}}
    \expandafter\def\csname rm#1\endcsname{\mathrm{#1}}
  }
 \expandafter\command@factory\next
}
\makeatother

\newcommand*{\longhookrightarrow}{\ensuremath{\lhook\joinrel\relbar\joinrel\rightarrow}}
\newcommand {\tto}{ \to\rangle }
\newcommand {\onto} {\twoheadrightarrow}
\newcommand {\into} {\hookrightarrow}
\newcommand {\xra} {\xrightarrow}    
\newcommand{\ra}{\rightarrow}
\newcommand{\imp} {\Rightarrow}
\newcommand{\actedon}{\curvearrowleft} 
\newcommand{\actson}{\curvearrowright}

\newcommand {\sd} {\rtimes}   
\newcommand{\semidirect}{\ltimes}
\newcommand{\isemidirect}{\rtimes}
\newcommand{\tensor}{\otimes}
\newcommand{\wreath}{\Lbag}

\newcommand{\du}{\sqcup}
\newcommand{\Dunion}{\bigsqcup} 
\newcommand{\disjoint}{\sqcup}
\newcommand{\normal} {\vartriangleleft}

\newcommand {\ie}{ i.e.\  }

\newcommand{\ul}[1]{\underline{#1}} 
\newcommand{\ol}[1]{\overline{#1}}


\newcommand{\Cay}{\operatorname{Cay}}
\newcommand{\proj}{\operatorname{proj}}
\newcommand{\Fix} {\operatorname{Fix}}
\newcommand{\dist}{\operatorname{dist}}
\newcommand{\diam}{\mathop{\mathrm{diam}\;}}

\newcommand{\yam}{$\bullet$}
\newcommand{\modif}{$ \clubsuit$}
\newcommand{\coucou}[1]{\footnote{#1}\marginpar{$\leftarrow$}}
\newcommand{\needref}{\textsuperscript{{\it citation needed}}\marginpar{$\leftarrow$}}
\newcommand{\why}{\textsuperscript{{\it Why ?}}\marginpar{$\leftarrow$}}

\title{Horospheres in degenerate 3-manifolds}

\author[Cyril Lecuire]{Cyril Lecuire}

\address{
Universite Paul Sabatier,
118 Route de Narbonne, 31400 Toulouse, France}

\email{lecuire@math.univ-toulouse.fr}

\author[Mahan Mj]{Mahan Mj}

\address{School of Mathematics,
Tata Institute of Fundamental Research. 1, Homi Bhabha Road, Mumbai-400005, India}

\email{mahan@math.tifr.res.in}

\subjclass[2010]{30F40, 57M50}

\thanks{The research of the first author was partially supported by the ANR grant GDSOUS.
The research of the second author is partially supported by  a DST J C Bose Fellowship. }   

\date{\today}

 \begin{abstract}
 We study horospheres in hyperbolic 3-manifolds $M$ all whose ends are degenerate. Towards this, we study
 which almost minimizing
geodesics in $M$ go through
arbitrarily thin parts. 
\end{abstract}

\maketitle

\tableofcontents

\section{Introduction} The topological study of unipotent orbits in locally symmetric spaces
of finite volume  has a rich history  with some of the highlights
being work of Hedlund \cite{hedlund}, Margulis \cite{margulis-opp}
and Ratner \cite{ratner-top}. However, the study in quotients of symmetric spaces by discrete subgroups
of infinite covolume (c.f.\ \cite{sarnak-thin}) is in its infancy except in the special case of rank one
symmetric spaces, or more generally infinite volume manifolds of pinched negative curvature, where 
the initial work was done by Eberlein \cite{eberlein1, eberlein2}. Following up from work of Eberlein,
a lot of work was done in dimension two, where one is interested in particular in
the behavior of the horocycle flow
on negatively curved surfaces of infinite genus (c.f.\ \cite{ds, haas, sarig}). In this paper we initiate the detailed
study of unipotent orbits (or horospheres) in hyperbolic 3-manifolds of infinite volume. We are particularly 
interested in degenerate hyperbolic 3-manifolds, i.e.\ hyperbolic 3-manifolds all whose ends are degenerate. Equivalently a degenerate hyperbolic 3-manifold is a quotient $M=\bH^3/G$ where $G$ is a finitely generated discrete subgroup of ${\rm PSL}_2(\C)$ which is not a lattice and the limit set $\Lambda$ of $G$ is the whole sphere $\bS^2$. We will assume that our manifolds have no parabolics. We are following here the terminology
introduced by Thurston in \cite{thurstonnotes} (as opposed to groups occurring in the Maskit slice, which
have been called by similar names in the literature).

The following characterization is due to Eberlein \cite{eberlein2}, Ledrappier \cite{ledrappier} and Coud\`ene-Maucourant \cite{cm}, (we will explain more precisely who proved what in sections 2.1 and 2.2)

\begin{theorem}\label{prel} (see Theorem \ref{omni})
 Let $M$ be a degenerate hyperbolic 3-manifold and let $\gamma (:= \gamma (t))$ be a geodesic ray in $M$
parametrized by arc length (and hence oriented). Let $W^{ss}(\gamma (t))$ be the strong stable manifold (i.e.\ stable horosphere)
through $\dot\gamma(t)$.
 Then 
\begin{enumerate}
\item $W^{ss}(\gamma (t))$ is dense in $M$ if and only if $\gamma$ is {\bf not} almost minimizing.
\item $W^{ss}(\gamma (t))$ is recurrent but not dense in $M$
 if and only if $\gamma$ is thin and almost minimizing.
\item $W^{ss}(\gamma (t))$ is properly embedded in $M$
 if and only if $\gamma$ is thick and almost minimizing.
\end{enumerate}
\end{theorem}

A word about the terminology. 
We shall call a geodesic ray {\bf exiting} if it is properly embedded, i.e.\ it is not contained in a compact set.
A geodesic in $M$ is {\bf almost minimizing} if $|d_M(\gamma(0), \gamma(t)) - t|$
is uniformly bounded. Clearly, almost minimizing geodesics are exiting. 
Almost minimizing  geodesics were studied by Haas \cite{haas} in the context of `flute surfaces', certain
planar hyperbolic
surfaces with infinitely many cusps. A version of Theorem \ref{prel} was deduced by Dal'bo and Starkov
\cite{ds} in the context of infinitely geneated Schottky groups. Here we shall study 
almost minimizing  geodesics  in 3-manifolds.
We will define recurrent horospheres in \textsection \ref{nondense}, for now, let us just say that they are not properly embedded. A geodesic ray is {\bf thin} if it goes through 
points in $M$ of arbitrarily small injectivity radius, and is {\bf thick} otherwise. 
Given the conclusions of Theorem \ref{prel} above, studying horospheres in degenerate hyperbolic
3-manifolds boils down to studying the following question:

\begin{pb}\label{maintt1}
Describe all almost minimizing
geodesics. 
\end{pb}

Problem \ref{maintt1} has quite a satisfactory solution in terms of a necessary and
sufficient condition. The ideal point of a lift of an almost minimizing geodesic ray lies in the complement
 $\Lambda^c_H (:= \Lambda \setminus \Lambda_H)$  of the horospherical limit set $\Lambda_H$. Another characterization can be made in term of the number of preimages under a Cannon-Thurston map. Given $M=\bH/G$, let $i : \Gamma_G \to \bH^3$ 
be the map that naturally comes from identifying the vertices of a Cayley graph of $G$
with the orbit of a point in $\bH^3$. Since we have assumed that $G$ has no parabolics, $\Gamma_G$ is hyperbolic. A {\em Cannon-Thurston map} is  the restriction to the ideal boundary $\partial\hat i:\partial\hat\Gamma_G\to\bS^2$ of a continuous extension $\hat i:\hat\Gamma_G\to\bD^3$ of $i$. The existence and structure of such a map has been studied in \cite{mahan-split, mahan-kl}.   Let $\Lambda_m$ denote the multiple 
limit set, i.e.\ the collection of
points in the limit set $\Lambda$ that have more than one pre-image under the Cannon-Thurston map. The conclusion of Section \ref{hll} can be summarized as follows:

\begin{theorem} \label{hllmain}
$\Lambda_m=\Lambda_H^c$ is the set of ideal points of lifts of almost minimizing geodesic rays. 
\end{theorem}

Theorem \ref{hllmain} answers an issue that has come up in works of several authors \cite{kapovich-ahlfors, gerasimov, jklo} who tried to relate the injective points of the Cannon-Thurston map to the {\it conical}
limit set. They concluded that the conical limit set is strictly contained in the set of 
injective points of the Cannon-Thurston map. Theorem \ref{hllmain} thus shows that,
in characterizing the injective points of the Cannon-Thurston map,
 the right limit set to be looking
at is the {\it horospherical} rather than the conical limit set.

Then we are led to the following:

\begin{pb}\label{maintt2}
Give conditions to determine which almost minimizing
geodesics in a degenerate end of a hyperbolic 3-manifold are thick and which are thin. 
\end{pb}

At this juncture, a kind of Murphy's Law breaks loose:

\begin{quote} {\it Anything that can go wrong does go wrong.} \end{quote}

It seems difficult to solve Problem  \ref{maintt2} comprehensively
and we find a number of examples and counterexamples. 
At the end of Section \ref{dense-non-dense} we give two examples:
one in which all almost minimizing geodesics are thick and one in which all 
almost minimizing geodesics are thin. What these examples bring out is the importance of
`building blocks' in trying to solve Problem  \ref{maintt2}. 
Thus from Section
\ref{models} onwards, we attempt  to address Problem \ref{maintt2} in terms of  
  the model geometry of ends
\cite{minsky-elc1, minsky-elc2, mahan-ibdd, mahan-split}. 

A number of model geometries for degenerate ends of hyperbolic 3-manifolds have come up, based primarily
on Minsky's monumental work \cite{minsky-jams, minsky-torus, minsky-elc1}  culminating in the resolution of
the Ending Lamination Conjecture by Brock-Canary-Minsky in \cite{minsky-elc2}, and in
the second author's proof of the existence of Cannon-Thurston maps \cite{mahan-split, mahan-kl}. 
In increasing
order of complexity, these are:

\begin{enumerate}
\item Bounded geometry \cite{minsky-jams, mitra-trees, mahan-bddgeo}
\item $i-$bounded geometry \cite{minsky-torus, ctm-locconn, mahan-ibdd}
\item Amalgamation geometry \cite{mahan-amalgeo}
\item Split geometry \cite{mahan-split}
\end{enumerate}

Of these, the first three are special
and are the subject of study in Section \ref{models},
 while every degenerate end $E$ does admit a model of split geometry.
The classification of these geometries depends on geometries of `building blocks', i.e.\
geometries of copies
of (topological) product regions $S \times I$ that are glued end-to-end to build up $E$.
Our explorations lead to the following conclusions in the special cases
of Section \ref{models}. As the reader 
will note, the conclusions become weaker and weaker as complexity increases.

\begin{theorem}\label{omni2} (See Lemma \ref{bddgeo}, Proposition \ref{ibddgeo} and Lemma \ref{amalgeo-suff}.)
\begin{enumerate} 
\item Let $E$ be of bounded geometry. Then every exiting geodesic is thick. In particular
every almost minimizing geodesic is thick. 
\item  Let $E$ be of i-bounded geometry. Then there exist thin exiting geodesics. However,
every almost minimizing 
geodesic is thick. 
\item 
Any almost minimizing geodesic in an amalgamated geometry end is thick if  all amalgamated
blocks have bounded thickness. 
\end{enumerate}
\end{theorem}

Here `thickness' (roughly) refers to the shortest distance between the bottom and top surfaces
(i.e.\ $S \times \{ 0 \}$ and $S \times \{ 1 \}$). 
We note here  that the  sufficient condition of `bounded thickness'
in Item (3) of Theorem \ref{omni2} is quite strong.

To proceed further (and deal with the general case of split geometry)
 a fair bit of technical material from \cite{masur-minsky2, minsky-elc1, mahan-split} is necessary. So as not to interrupt the flow of the paper,
 we proceed assuming this and relegate a summary of the relevant material to an Appendix, Section \ref{splitt}.

Section \ref{ctreg} dwells on counterexamples.
Using  the technology of  Section \ref{splitt}, we find that the
 condition of `bounded thickness'
in Item (3) of Theorem \ref{omni2}
is not a necessary condition even in the special case of amalgamated
geometry ends. 
The example of Section \ref{unbddth} shows that
it is possible to have thick almost minimizing geodesics in manifolds of amalgamation
geometry even in the presence of arbitrarily thick
amalgamation blocks.
Further, in the general case of split geometry, the sufficient condition of `bounded thickness'
is neither necessary nor sufficient. 
In Section \ref{ctreg-split} we  provide a counterexample to show that there does
exist an end of split geometry, where all the building blocks (or `split blocks') have bounded
thickness
but almost minimizing geodesics are thin.

The examples of Section \ref{ctreg}
 seem to justify the `Murphy's Law' that we mentioned above: as we progress to greater
degrees of complexity of the geometry of ends, we tend to lose any hope of systematically
 characterizing which almost minimizing geodesics are thick and which are thin, i.e.\ we are
unable to provide a satisfactory answer to Question
\ref{maintt2}   in the 
most general case (of split geometry). More precisely, the counterexamples in Section \ref{ctreg}
show that the natural property of {\bf thickness} of blocks fails to detect thickness or thinness
of almost minimizing geodesics.

In hindsight, the difficulty in answering Question
\ref{maintt2} manifests in the difference in the approaches of McMullen \cite{ctm-locconn} and the second author
\cite{mahan-split} in proving the existence of Cannon-Thurston maps, i.e.\ $\pi_1(S)-$equivariant 
continuous maps from the
(hyperbolic or relatively hyperbolic) boundary of $\pi_1(S)$ onto the limit set. McMullen finds, in the special case of
a punctured torus Kleinian group,  precise locations (in $\til{E}$, the universal cover of an end) of geodesics parametrized
by the boundary at infinity $S^1_\infty$ (of $\pi_1({S})$). On the other hand, the approach 
to the general case of split geometry in \cite{mahan-split} necessarily forgets
much of the fine structure contained within the building (split) blocks by `electrocuting' their connected components. Punctured torus
groups provide an example of $i-$bounded geometry \cite{mahan-ibdd}, which is quite special. As such an analog
of McMullen's approach in the general case of split geometry is missing. The counterexamples in Section \ref{ctreg}
indicate that even answering Question \ref{maintt2}, which is a small component of the more general problem of finding
precise locations of geodesics in $\til{E}$, is tricky in general. 

\section{Dense and non-dense horospheres}	\label{dense-non-dense}
We start with a few definitions. Let $M$ be a complete hyperbolic $3$-manifold, $\til{M}=\bH^3$ its universal cover
and $\partial \til{M}$ the boundary at infinity of $\til M$. We will use the unit ball model and then $\partial\bH^n$ is the unit sphere $\bS^{n-1}$. We identify $\pi_1(M)$ with the group of deck-transformations on $\til M$. \\

Let $SM$ be the unit tangent bundle of $M$ and $S\til M$ the unit tangent bundle of $\tilde M$. Denote by $g:\reals\times S\tilde M\to S\tilde M$ the geodesic flow. The {\bf strong stable manifold} or {\bf stable horosphere} through $\tilde v\in S\tilde M$ is the set $W^{ss}(\tilde v) = \{\tilde w\in S\tilde M: d(g(t,\tilde v),g(t, \tilde w)) \longrightarrow 0 \, {\rm when} \, t\longrightarrow 0\}$. The stable horosphere  $W^{ss}(v)\subset SM$ through a vector $v\in SM$ is the projection of the stable horosphere through a lift of $v$ to $\tilde v$.

In the hyperbolic space $\bH^n$ of dimension $n$, a horosphere ${\mathcal H}$ is the intersection of $\bH^n$ with a round sphere tangent to the boundary at infinity (in the unit ball model of $\bH^n$), i.e.\ the round sphere minus the point of tangency. A stable horosphere is the set of unit normal vectors to a horosphere ${\mathcal H}$ pointing into the round ball bounded by ${\mathcal H}$.

As stated in the Introduction, we call a hyperbolic 3-manifolds $M$ degenerate if all its ends are degenerate.
We will define the ends of 3-manifold in section \ref{ends}, for now, let us give an alternate definition. For our
purposes, a hyperbolic 3-manifold is a quotient $M=\bH^3/G$ where $G$ is a finitely generated discrete subgroup of ${\rm PSL}_2(\C)$. The manifold $M$ is {\bf degenerate} if $G$ is not a lattice and its limit set $\Lambda$ is the whole sphere $\bS^2$.

\medskip

\noindent {\bf Convention:} Unless otherwise mentioned,
all degenerate manifolds in this paper
will be without parabolics. $M$ will denote a degenerate hyperbolic 3-manifold and $G$ the corresponding
Kleinian group.

\subsection{Almost minimizing geodesics and dense horospheres} 

In this section we will expose results of Eberlein relating dense horospheres with almost minimizing geodesic rays. We start with some definitions and properties of almost minimizing geodesic rays.
\begin{defn}
Given $C\geq 0$, a geodesic ray $\gamma = \gamma(t): t \in \reals^+$ in $M$  is called {\bf 
$C$-almost minimizing} if it is has unit speed and
  $ d_M(\gamma(0),\gamma(t)) \geq t- C$ for any $t\in\reals^+$. A geodesic ray is {\bf almost minimizing} if it is $C$-almost minimizing for some $C\geq 0$.

A geodesic ray $\gamma = \gamma(t): t \in \reals^+$,  is called {\bf asymptotically
almost minimizing} if it has unit speed  and if for any  $\delta > 0$ there exists $T$
such that for any $s, t \in \reals^+$ with $ s, t \geq T$, $ d_M(\gamma(s),\gamma(t)) \geq |s-t|-\delta$.

A point $\xi \in \partial \til{M}$ is a horospherical limit point of $\pi_1(M)$, if for any base-point $o\in \til{M}$
and any horoball $B_\xi$ based at $\xi$, there exist infinitely many translates $g.o \in B_\xi$,
where $g \in \pi_1(M)$.
The collection of horospherical limit points of $\pi_1(M)$ is called the horospherical limit set
$\Lambda_H$ of $\pi_1(M)$.
\end{defn}

These definitions are related by the following combination of results of Eberlein and Ledrappier, who consider the much more general context of complete manifolds of pinched negative curvature.

\begin{prop}\cite[Proposition 4]{ledrappier},\cite[Proposition 7.4]{eberlein-visibility}
\label{ameqnt}
Let $\gamma$ be a geodesic ray in a negatively curved manifold $M$. The following are equivalent:
\begin{enumerate}
\item $\gamma$ is almost minimizing
\item $\gamma$ is asymptotically almost minimizing
\item $\tilde\gamma (\infty) \in (\partial \til{M} \setminus \Lambda_H)$ 
\end{enumerate}
\end{prop}

In \cite[Theorem 5.2]{eberlein1}, Eberlein shows the existence of dense horospheres for negatively curved manifolds satisfying Axiom 1 (any $2$ points in the boundary at infinity of the universal cover are joined by at least one geodesic) for which the nonwandering set $\Omega$ is the whole unit tangent bundle $SM$. 
Manifolds satisfying Axiom 1 are called {\bf visibility manifolds}.
Complete negatively curved manifolds of pinched negative curvature are examples of visibility manifolds.
Hyperbolic manifolds are, therefore,  visibility manifolds. Furthermore, classical results results imply that when $M=\bH^n/G$ then $\Omega=SM$ if and only if $\Lambda_G=\partial_\infty\bH^n$. Thus if $M$ is a hyperbolic 3-manifold
with finitely generated fundamental group, then $\Omega=SM$ if and only if $M$ is either degenerate or 
has finite volume. We state the next two Theorems, due to Eberlein, in their full generality but for our purpose
in this paper, the reader can replace "Let $M$ be a negatively curved visibility manifold such that $\Omega=SM$" with "Let $M$ be a degenerate hyperbolic 3-manifold".

\begin{theorem}\label{densev}
Let $M$ be a negatively curved visibility manifold  such that $\Omega=SM$. Then there exists a vector $v \in SM$, such that
the strong stable manifold  $(W^{ss}(v))$ is dense in $SM$.
\end{theorem}

Going further \cite[Theorem 5.5]{eberlein1}
Eberlein relates the density of horospheres to almost minimizing geodesic rays: 

\begin{theorem}\label{am}
Let $M$ be a negatively curved visibility manifold  such that $\Omega=SM$. Then  $(W^{ss}(v))$ is dense in $SM$
if and only  if $v$ is {\bf not} almost minimizing.
\end{theorem}

Any degenerate hyperbolic 3-manifold $M$ with infinite diameter has minimizing geodesic rays:  choose a sequence $p_n$ exiting any compact in $M$,
and, up to extracting a subsequence, take a  limit of minimizing geodesic segments $[o,p_n]$. Consequently, we have:

\begin{remark}\label{nondense-rem}
 Let $M$ be a degenerate hyperbolic 3-manifold. Then there exist minimizing geodesic rays $\gamma:\reals^+ \to M$ and hence
the horospheres $(W^{ss}(\dot\gamma(t)))$ are {\bf not} dense in $M$ for $t>0$. 
\end{remark}

\subsection{Thick and thin geodesics}
Next we want to discuss non-dense horospheres and hence almost minimizing geodesic rays. Works of Ledrappier (\cite{ledrappier}) and Coud\`ene-Maucourant (\cite{cm}) relate proper and recurrent horospheres with thick and thin geodesic rays. Let us first introduce thick and thick geodesic rays.

The injectivity radius of $M$ at a point $x\subset M$ is the maximal radius of an embedded ball centered at $x$. Let $p:\tilde M\to M$ be the covering projection, let $\tilde x\subset \tilde M$ be a lift of $x$ and let $B(\tilde x,r)$ be the ball with diameter $r$ centered at $\tilde x$. Then the injectivity radius at $x$ is:

$${\rm Inj}(x)=\sup\{r|p_{|B(\tilde x,r)}{\rm \, is \, an \,isometry}\}.$$

\begin{defn} A geodesic $\gamma:\reals^+ \to M$ is said to be {\bf thick} if 
$$liminf_t Inj {(\gamma(t))} >0.$$
Otherwise it is called {\bf thin}.
\end{defn}

It is easy to see that for a geodesic ray, the property of being almost minimizing, thin or thick (and exiting which will be defined later on) depends only on its ideal endpoint. In other word given two geodesic rays $\gamma_1,\gamma_2 : [0, \infty)\to M$ which have asymptotic lifts to $\tilde M$, then $\gamma_1$ is  thin, thick, exiting or almost minimizing if and only if $\gamma_2$ has the same property. Thus when dealing with these properties, it seems appropriate to parametrize geodesic rays by their endpoints or the endpoints of their lifts. For future reference, let us show this fact for almost minimizing ray.

\begin{lemma}	\label{minimizing-almost-minimizing}
A geodesic ray which has a lift to $\tilde M$ that is asymptotic to a lift of an almost minimizing geodesic is almost minimizing.
\end{lemma}

\begin{proof}
Assume that a lift $\tilde\gamma$ of a geodesic ray $\gamma\subset M$ is asymptotic to a lift $\tilde\gamma'$ of a $C$-minimizing geodesic $\gamma'$. Then by
convexity of the distance between $2$ geodesics, we have $d(\tilde\gamma(t),\tilde\gamma'(t))\leq d(\tilde\gamma(0),\tilde\gamma'(0))$
for any $t>0$. Projecting to $M$, we get $d_M(\gamma(t),\gamma'(t))\leq d_M(\tilde\gamma(0),\tilde\gamma'(0))$  for any $t>0$. The
triangle inequality gives us $|d_M(\gamma(t),\gamma(0))-d_M(\gamma'(t),\gamma'(0))|\leq 2 d_M(\gamma(0),\gamma'(0))$. Thus we get $d_M(\gamma(t),\gamma(0))\geq t-C-2 d_M(\gamma(0),\gamma'(0))$.
\end{proof}

\begin{rem} \label{rem-minam}
	On the other hand, it is easy to construct two geodesic rays with asymptotic lifts such that one is minimizing and the 
other one is not (for example by adding a geodesic loop at the initial point of the minimizing ray and then straightening).
Thus almost minimizing geodesics depend only on the end-point on the sphere at infinity, while minimizing geodesics 
depend on the initial point also. This is the reason why we deal with almost minimizing rather than minimizing
geodesics in most of this paper.
\end{rem}

When there is a positive lower bound on the injectivity radius at any point of $M$, then $M$ is said to have {\bf bounded geometry} and obviously, every geodesic ray is thick. Otherwise $M$ is said to have {\bf unbounded geometry} and the situation is almost opposite, i.e.\ almost every geodesic ray is thin:

\begin{lemma} \label{nodich} Let $M$ be a degenerate hyperbolic 3-manifold with no positive lower bound
on injectivity radius. Then the preimage in the universal cover of the set of thin geodesic rays (parametrized by $S^2_\infty$)
 is of full measure (for the Lebesgue measure).
\end{lemma}

\begin{proof}
For every $e>0$, let $M_e \subset M$ be the set of points having injectivity radius at most $e$
and let $\til{M_e}$ be the full pre-image of $M_e$ in $\til M$ under the (universal) covering map.
Fix a base-point $o \in (\til{M} \setminus \til{M_e})$. Recall \cite{sullivan-pihes, sullivan-rigidity}
that the {\bf shadow}
 of a subset $K \subset \til{M}$ is the subset of $S^2_\infty (=\partial \Hyp^3)$ given by $$Sh(K) = \{ x \in S^2_\infty:
[o,x) \cap K \neq \emptyset \}.$$

Since $\til{M_e}$ is invariant under $\pi_1(M)$, its
 shadow $Sh(\til{M_e})$  is also invariant under the action of $\pi_1(M)$ on $S^2_\infty$. 
Next,  the action of $\pi_1(M)$ on $S^2_\infty$ is ergodic \cite{minsky-elc2}. Also, since $Sh(\til{M_e})$ contains
the shadow of a lift of a single Margulis tube, it follows that $Sh(\til{M_e})$ is of positive measure.
By ergodicity, $Sh(\til{M_e})$
is of full measure. We take intersections over $e_n = 1/n$ to conclude that thin geodesics are of full measure.

The last statement follows.
\end{proof}

Notice that a closed geodesic is always thick and hence $M$ always contains thick geodesic rays. As we will see later, degenerate hyperbolic 3-manifold with 
unbounded geometry always contain thick exiting geodesic rays (though the collection of such rays has measure zero
by Lemma \ref{nodich}).
 On the other hand, there are hyperbolic surfaces without cusps for which all exiting geodesic rays are thin. Let us construct an example:

\begin{lemma} \label{simpledich} There exists
 a complete hyperbolic surface $\Sigma$ of infinite genus such that the following dichotomy holds 
for geodesic rays $\gamma$ on  $\Sigma$:

\begin{enumerate}
\item Either $\gamma$ lies inside a compact set
\item or $\gamma$ is thin.
\end{enumerate} 
\end{lemma}

\begin{proof}
Let $T$ be a torus with two holes.
Let $T_n$ ($n > 0$) be a hyperbolic structure on $T$ such that the boundary components
are totally geodesic and have length $\frac{1}{n}, \frac{1}{n+1}$. Also let $T_0$ be a hyperbolic
torus with one boundary component of length $1$. Attach $T_i$ to $T_{i+1}$ by an isometry
between the boundary components of length $\frac{1}{i+1}$. We call the resulting
geodesic the $(i+1)-$th neck.  Let $\Sigma = \bigcup_0^\infty T_i$ be the union
modulo this identification,
such that the neck between the $n-$th and $(n+1)-$th torus summands has length $\frac{1}{n+1}$. Since
any $i-$th neck disconnects $\Sigma$ into a compact piece and a noncompact piece, it follows that any geodesic 
ray 
in $\Sigma$ is either bounded or cuts every neck and is therefore  thin. 
\end{proof}

\subsection{Non-dense Horospheres}\label{nondense}
We are now ready to describe the behavior of horospheres corresponding to almost minimizing rays.
Let $\Pi : \til{M} \rightarrow M$ be the covering projection.

\begin{defn} Let $\gamma := \gamma(t)$ be a geodesic in $M$. Then the stable horosphere $W^{ss}(\dot\gamma(t))$
is said to be {\bf recurrent} if the following holds:\\ Let
 $\dot\gamma_1(t)$ be a lift of $\dot\gamma(t)$ to the unit tangent bundle $S\til{M}$. Then for every compact $K \subset 
W^{ss} (\dot\gamma_1(t)) (\subset S\til{M})$, there exists a vector $w \in (W^{ss} (\dot\gamma_1(t)) \setminus K)$ such that 
its projection (under $\Pi$) is arbitrarily close to $\dot\gamma_1(t)$.\end{defn}

Notice that this seems to be a property of $\dot\gamma(t)$. But using the transitive action of the horocyclic flow, it is not hard to see that when the property above holds for $\dot\gamma(t)$ it holds for any vector $v\in W^{ss}(\dot\gamma(t))$. Thus it is a property of the stable horosphere. 

A relation between thick geodesic rays and embedded horospheres was first established by Ledrappier in \cite{ledrappier}:

\begin{theorem} \cite[Proposition 3]{ledrappier}
Let $M$ be a negatively curved manifold with bounded
geometry, $\gamma = \gamma(t): t \in \reals^+$ an asymptotically almost minimizing geodesic. Then the strong stable leaf
 $(W^{ss}(\dot\gamma(t)))$ is properly embedded in $M$ for $t > 0$.
\end{theorem}

Ledrappier's proof (c.f.\  \cite[Lemma 3]{ledrappier}) only uses the fact that
 the injectivity radius $Inj (\Pi(\delta(t))$ is bounded away from zero {\bf along} $\delta(t)$ (for $t>0$) for any geodesic ray $\delta$ such that $\delta(t)\subset W^{ss}(\dot\gamma(t))$ (for some and hence all $t$). As was noticed earlier this is equivalent to having the injectivity radius $Inj (\Pi(\gamma(t))$ bounded away from zero {\bf along} $\gamma(t)$  (for $t>0$).
 
Thus in the context of degenerate 3-manifolds, we can restate Ledrappier's result \cite[Proposition 3]{ledrappier}:

\begin{theorem}\label{am-injrad1}
 Let $M$ be a degenerate hyperbolic 3-manifold and let $\gamma = \gamma(t): t \in \reals^+$  be a thick almost minimizing geodesic ray. Then the strong stable leaf
 $(W^{ss}(\dot\gamma(t)))$ is properly embedded in $M$ for $t > 0$. 
\end{theorem}

A converse to Ledrappier's Theorem \ref{am-injrad1} is furnished by Coud\`ene and Maucourant
\cite[Section 3]{cm}, this time for complete manifolds with pinched negative curvature):

\begin{theorem}\cite[Section 3]{cm}	\label{am-injrad2}
Let $M$ be a complete manifolds with pinched negative curvature and let $\gamma = \gamma(t): t \in \reals^+$  be a thin geodesic ray. Then the strong stable leaf
 $(W^{ss}(\dot\gamma(t)))$ for $t>0$ is {\bf recurrent}. 
\end{theorem}

\subsection{Summary of section \ref{dense-non-dense}}
Combining Theorem \ref{am}, Remark \ref{nondense-rem} and  Theorems  \ref{am-injrad1}, \ref{am-injrad2}, we get:

\begin{theorem} \label{omni} Let $M$ be a degenerate hyperbolic 3-manifold and let $\gamma$ be a geodesic ray in $M$.
 Then 
\begin{enumerate}
\item $W^{ss}(\gamma (t))$ is dense in $M$ if and only if $\gamma$ is {\bf not} almost minimizing.
\item $W^{ss}(\gamma (t))$ is recurrent but not dense in $M$
 if and only if $\gamma$ is thin almost minimizing.
\item $W^{ss}(\gamma (t))$ is properly embedded in $M$
 if and only if $\gamma$ is thick almost minimizing.
\end{enumerate}
\end{theorem}

This leads us to the study of Questions \ref{maintt1} and \ref{maintt2} mentioned in the Introduction.
In the next two subsections, we use this result to furnish two sets of examples. In the first, all almost 
minimizing geodesics are thick and in the second, 
all almost 
minimizing geodesics are thin.
In section \ref{hll} we will address Question \ref{maintt1}. In the rest of the paper we will come back to Question \ref{maintt2} and address it in greater
 generality, relating it to different model geometries.
 
\subsection{Hyperbolic Dehn filling}
In the next two subsections, we will construct examples of thin manifolds by using the following version of Thurston's Hyperbolic Dehn Filling Theorem:

\begin{theorem}		\label{dehn-filling}
Let $M$ be a geometrically finite hyperbolic $3$-manifold whose convex core has totally geodesic boundary. Such 
a manifold $M$ is homeomorphic to the interior of a compact manifold $\hat M$. The rank $1$ cusps of $M$ correspond to a pants decomposition $R$ of the union $\partial_{\chi <0}$ of the components of $\partial \bar M$ with negative Euler characteristic. Let $T_0, ... , T_q$ be torus components of $\partial M$ and let $\bar M(p_0,...,p_q)$ be the manifold obtained by performing $(1,p_i)$ Dehn filling on $T_i$, $i=0, \cdots, q$.
 Then for $p_0,...,p_q$ large enough the interior of $\bar M(p_0,...,p_q)$ admits a unique geometrically finite hyperbolic metric with totally geodesic boundary $M(p_0,...,p_q)$ such that the rank $1$ cusps correspond to $R$. Furthermore $M(p_0,...,p_q)$ converges geometrically to $M$ when $(p_0,...,p_q)\longrightarrow (\infty,...,\infty)$.
\end{theorem}

\begin{proof}
Let $D\bar M$ be the compact $3$-manifold obtained by gluing $2$ copies of $\bar M$ along $\partial_{\chi<0}\bar M$ and removing a regular neighborhood of $R\subset \partial_{\chi<0}\bar M$. The interior of $M$ admits a complete hyperbolic metric with finite volume obtained by gluing $2$ copies of the convex core of $M$ along their boundaries. Let $D\bar M(p_0,...,p_q)$ be the compact $3$-manifold obtained by gluing $2$ copies of $\bar M(p_0,...,p_q)$ along $\partial_{\chi<0}\bar M(p_0,...,p_q)$ and removing a regular neighborhood of $R\subset \partial_{\chi<0}\bar M(p_0,...,p_q)$. It follows from Thurston's Dehn Filling Theorem \cite[Theorem 5.8.2]{thurston-hypstr2} that for $p_0,...,p_q$ large enough the interior of $D\bar M(p_0,...,p_q)$ admits a hyperbolic metric with finite volume, which is unique up to isometries by Mostow-Prasad's Rigidity Theorem, let us denote by $DM(p_0,...,p_q)$ the resulting hyperbolic manifold. Again by  Mostow-Prasad's Rigidity Theorem the natural involution $\tau :D\bar M(p_0,.
..,p_q)\to D\bar M(p_0,...,p_q)$ which exchange the $2$ copies of $\bar M(p_0,...,p_q)$ is isotopic to an isometry. The quotient of $DM(p_0,...,p_q)$ by this isometry is the convex core of the desired hyperbolic manifold $M(p_0,...,p_q)$.

Still by Thurston's Dehn Filling Theorem (\cite[Chapter 5]{thurston-hypstr2}, see also \cite{petronio-porti}), $DM(p_0,...,p_q)$ converges geometrically to $DM$ when $(p_0,...,p_q)\longrightarrow (\infty,...,\infty)$ and $M(p_0,...,p_q)$ converges geometrically to $M$.
\end{proof}

Although we did not a find a statement containing the Theorem above in the literature, this could certainly be deduced from previous work such as \cite{bon-otal} or \cite{bromberg}.

\subsection{Example: A thin manifold  all of whose almost minimizing geodesics are thick}	\label{thick}

  The first example is due to Thurston \cite{thurston-hypstr2}
and Bonahon-Otal \cite{bon-otal}.
 We will detail a construction explained in \cite{thurston-hypstr2} and show that the result has the expected property. Later, when we describe i-bounded geometry as a model geometry of ends,
it will become clear that these examples form a special case. However, since the examples in this section
can be described in  a reasonably self-contained manner, we explicitly describe these below.

Let $S$ be a closed surface and $P,Q\subset S$ two pants decompositions that fill $S$, i.e. the connected component of $S \setminus (P\cup Q)$ are discs.. Consider a faithful and discrete representation $\pi_1(S)\to PSL(2,\C)$ whose convex core $C(P,Q)$ has totally geodesic boundary and cusps corresponding to $P$ on the bottom side and to $Q$ on the top side.

Let $\hat M_0$ be the manifold obtained by gluing $C(P,Q)$ on top of $C(Q,P)$ and let $\hat M_n$ be the manifold obtained by gluing $2n+1$ copies of $\hat M_0$ on top of each other. Pick a base point $x_n$ in the middle piece of $\hat M_n$. By construction, $\hat M_i$ isometrically embeds in $\hat M_j$ for any $j>i$. It follows that the sequence $(\hat M_n,x_n)$ converges geometrically to a hyperbolic $3$-manifold $\hat M_\infty$.

Next we fill the holes of $\hat M_\infty$ recursively. Let $M_0(p_0)$ be the manifold obtained by performing $(1,p_0)$ Dehn fillings on the torus cusps of $M_0$ (c.f.\ Theorem \ref{dehn-filling}). Given $M_n(p_0,...,p_n)$, glue a copy of $\hat M_0$ at the top and one at the bottom to obtain a new convex hyperbolic $3$-manifold $\hat M_{n+1}(p_0,...,p_n)$. Denote then by $M_{n+1}(p_0,...,p_{n+1})$ the manifold obtained by performing $(1,p_{n+1})$
 Dehn fillings on the torus cusps of $\hat M_{n+1}(p_0,...,p_i)$ (c.f.\ Theorem \ref{dehn-filling}). Also denote by $\hat M_\infty(p_0,...,p_n)$ the manifold obtained by gluing $\hat M_\infty-\hat M_n$ along the boundary of $M_n(p_0,...,p_n)$ (or equivalently perform $(1,p_i)$ Dehn fillings on the appropriate cusps of $\hat M_\infty$).

If we fix $n$ and $p_0,...,p_{n-1}$ and let $p_n$ go to $\infty$, by Theorem \ref{dehn-filling}, $M_n(p_0,...,p_n)$ converges geometrically to 
$\hat M_n(p_0,...,p_{n-1})$.
 It follows that for $p_n$ large enough (depending on $\eps$) there is a homeomorphism $f_n$ between the $\eps$-thick parts of  $\hat M_\infty(p_0,...,p_n)$ and $\hat M_\infty(p_0,...,p_{n-1})$ whose restriction to $\hat M_\infty-\hat M_n$ is an isometry and restriction to  $M_n(p_0,...,p_n)$ is  $K_n(p_n)$-bilipschitz, with $K_n(p_n)$ close to $1$ when $p_n$ is large. This also implies that the sum $l_n(p_n)$ of the lengths  of the added geodesics is short when $p_n$ is large.

Pick a small $\eps$ and choose the sequence $\{p_n\}$ so that $\Pi_n K_n(p_n)$ converges and that $l_n(p_n)\longrightarrow 0$. The map $g_n=f_1\circ...\circ f_n: \hat M_\infty(p_0,...,p_n)\to \hat M_\infty$ is bilipschitz on the $\eps$-thick part. It follows that $(M_n(p_0,...,p_n), x_n)$ and $(\hat M_\infty(p_0,...,p_n),x_n)$ converge to a hyperbolic manifold $M_\infty$ homeomorphic to $S\times\R$ and that $g_n$ converge to a bilipschitz map $g_\infty$ from the $\eps$-thin part of $M_\infty$ to the $\eps$-thin part of $\hat M_\infty$. Also since $p_n \to \infty$, the injectivity radius of $M_\infty$
has no positive lower bound. 

By construction, each cusp of $\hat M_\infty$ is isometric to a cusp of $\hat M_1$. In particular the components of the boundary of the $\eps$-thick part of $\hat M_\infty$ have uniformly bounded diameter. Then the map $g_\infty$ provides us with an upper bound $D$ on the diameters of the components of the boundary of the $\eps$-thick part of $M_\infty$. Let $\kappa$ be an arc in $M_\infty$ with endpoints in the $\eps$-thick part. If $\kappa$ goes through the $\eps_0$ thin part, it has a subsegment of length $l(\eps,\eps_0)$ in the $\eps$-thin part with $l(\eps,\eps_0)\longrightarrow\infty$ when $\eps_0$ tends to $0$. Hence $\kappa$ is not $(l(\eps,\eps_0)-D-1)$-minimizing.

We conclude that a geodesic that goes arbitrarily deep in the thin part of $M_\infty$ is not almost minimizing.
We shall generalize this example considerably in Section \ref{ibdd}.

\subsection{Example: A thin manifold  all of whose almost minimizing geodesics are thin}	\label{thin}
For the second example, we will follow the same procedure but the pieces we will glue will be different. Let $c\subset S$ be a non separating curve, $\phi:S-c\to S-c$ a pseudo-Anosov diffeomorphism and let $P$ be a pants decomposition that crosses $c$. We will use $C(P,\phi^j(P))$ with larger and larger $j$ instead of $C(P,Q)$. These pieces have the following property:

\begin{lemma}		\label{thin0}
Given $\eps,C$, there is $J=J(\eps,C,\phi,P)$ such that for any $j\geq J$ any $C$-almost minimizing segment joining the top boundary of $C(P,\phi^j(P))$ to its bottom boundary goes through the $\eps$-thin part.
\end{lemma}

\begin{proof}
Let $M_\phi$ be the hyperbolic manifold homeomorphic to $(S-c)\times [0,1]/(x,0)\sim (\phi(x),1)$ and let $\bar M_\phi$ be its cyclic cover (homeomorphic to $(S-c)\times (0,1))$). Basic hyperbolic geometry tells us that the distance between $2$ points on a horoball grows logarithmically with their distance on the horosphere. Applied to $\bar M_\phi$ this produces the following Claim:

\begin{claim}	\label{horo}
Pick a fundamental domain $D$ for the action of $\Z$ on $\bar M_{\phi}$. Given $\eps$, we denote by $D^k_\eps$ the union of $k$ adjacent copies of the $\eps$-thick part of $D$ in $\bar M_\phi$. Let $x\subset \bar M_\phi$ be a point at the top of $D^k_\eps$ and $y\subset \hat M_\phi$ be a point at the bottom of $D^k_\eps$, then for $k$ large enough $d_{\hat M_\phi}(x,y)\leq 2\log k$.\hfill $\Box$
\end{claim}

We will transport this property into $C(P,\phi^j(P))$ by showing that for $j$ large enough, a large part of $C(P,\phi^j(P))$ looks like $D^k_\eps$.

\begin{claim}	\label{emb}
Given $\eps, k$, there is $I=I(\eps,k,\phi,P)$ such that for $j\geq I$, there is a $1+\eps$-bilipschitz embedding of $D^k_\eps$ in $C(P,\phi^j(P))$. 
\end{claim}

\begin{proof}
Set $j=2i$ if $j$ is even and $j=2i+1$ otherwise. Let $\rho_i:\pi_1(S)\to PSL(2,\C)$ be a discrete and faithful representation with cusps corresponding to $\phi^{-i}(P)$ at the bottom and cusps corresponding to $\phi^i(P)$ at the top if $j$ is even and to $\phi^{i+1}(P)$ if $j$ is odd. Notice that the convex core of $\bH^3/\rho_i(\pi_1(S))$ is isometric to $C(P,\phi^j(P))$. Consider the restriction $\rho_{c,i}:\pi_1(S-c)\to PSL(2,\C)$ of $\rho_i$ to $\pi_1(S-c)$. By \cite{minsky-kleinian}, the length of the geodesic corresponding to $c$ in $\bH^3/\rho_i(\pi_1(S))$ tends to $0$ when $i$ tends to $\infty$. It follows from 
a generalization of the Double Limit Theorem (\cite{thurston-hypstr2}, see also \cite{canary-schottky}) that a subsequence of $\rho_{c,j}$ converges to a representation $\rho_\infty$. Since its length goes to $0$, $c$ is a parabolic for $\rho_\infty$. By \cite{brock-conti}, the stable and unstable laminations of $\phi$ are not realized in $\bH^3/\rho_\infty$. It follows from the Ending Lamination Theorem (\cite{minsky-elc2}) that $\bH^3/\rho_\infty(\pi_1(S-c))$ is isometric to $\bar M_\phi$. Up to extracting a further subsequence, $\bH^3/\rho_n(\pi_1(S-c))$ converges geometrically as well. By the Covering Theorem \cite{canary-cover}, $\bar M_\phi$ is also the geometric limit. The conclusion follows.
\end{proof}

Combining Claims \ref{horo} and  \ref{emb}, we see that distances grow linearly with $k$ in the thick part while they grow logarithmically in the thin part. Now we just need to adjust the constants to prove Lemma \ref{thin0}.

Fix $\eps$, $C$, $j$ and $k$ large enough so that Claim \ref{horo} holds for $\frac{\eps}{2}$. Set $\eps_k$ such that a geodesic segment in $\bar M_\phi$ joining $2$ points in $D^k_{\frac{\eps}{2}}$ does not enter the $\eps_k$ thin part. Let $d$ be the distance between the $2$ boundary components of $D$ in the $\frac{\eps}{2}$-thick part. Let $\kappa$ be a segment in the $\eps$-thick part joining the top boundary of $C(P,\phi^j(P))$ to its bottom boundary.

By Claim \ref{emb}, if  $j\geq J(k,\eps_k)$, there is a $2$-bilipschitz embedding $f:D^k_{\eps_k}\hookrightarrow C(P,\phi^j(P))$. The preimage $\eta$ of $\kappa\cap f(D^k_{\eps_k})$ joins the top of $D^k$ to its bottom and lies in the $\frac{\eps}{2}$-thick part. It follows that $\eta$ has length at least $kd$ and that $\kappa\cap f(D^k_{\eps_k})$ has length at least $k\frac{d}{2}$. By Claim \ref{horo}, the endpoints of $\eta$ are joined by an arc $\eta'\subset \hat M_\phi$ with length at most $2\log k$. Furthermore, by the choice of $\eps_k$, $\eta'\subset D^k_{\eps_k}$. Thus $f(\eta')$ is an arc with length at most $4\log k$ joining the endpoint of $\kappa\cap f(D^k_{\eps_k})$. Now we can conclude that if $k\frac{d}{2}> 4\log k+C$ and $j\geq J(k,\eps_k)$, $\kappa$ is not $C$-almost minimizing.
\end{proof}

The second example is now
constructed with the same steps as the first one. Let $M(j)$ be the manifold obtained by gluing $C(\phi^j(P),P)$ on top of $C(P,\phi^j(P))$. Define $M_1=M(1)$ and define $M_{n+1}$ recursively by gluing a copy of $M(n+1)$ at the top of $M_n$ and one at the bottom. Pick a basepoint $x_n$ in the middle piece of $M_n$. It is easy to see that $(M_n,x_n)$ converges geometrically to a hyperbolic $3$-manifold $\hat M_\infty$ with infinitely many rank $2$ cusps. It is easy to deduce from Lemma \ref{thin0} that any almost minimizing geodesic in $\hat M_\infty$ goes arbitrarily far into the thin part.

Now the manifold $M_\infty$ is obtained as in the first example by recursively performing $(1,p_i)$-Dehn filling on the cusps of $\hat M_i$, with an appropriate choice of $p_i$ so that everything converges and that the geometry is close to the geometry of $\hat M_\infty$. Then by Lemma \ref{thin0}, any almost minimizing geodesic in $M_\infty$ goes arbitrarily far into the thin part.

\section{The horospherical limit set}\label{hll}

In this section, we study the interrelationships between three subsets of the limit set:

\begin{itemize}
\item The conical limit set $\Lambda_c$.
\item The multiple limit set $\Lambda_m = \{ x \in \Lambda : \# |(\partial\hat{i})^{-1}(x)| > 1\}$,
where $\partial\hat{i}$ denotes the Cannon-Thurston map (defined below).
\item The horospherical limit set $\Lambda_H$.
\end{itemize}

Before recalling the definitions of these sets, let us consider the topology and geometry of hyperbolic 3-manifolds.

\subsection{Ends of hyperbolic 3-manifold}		\label{ends}
We
have mentioned earlier that, for our purposes,
a hyperbolic 3-manifold is a quotient $M=\bH^3/G$ where $G$ is a finitely generated Kleinian group. It follows from the Tameness Theorem (\cite{agol-tameness}, \cite{calegari-gabai}) that $M$ is {\bf tame}, i.e.\ homeomorphic to the interior of a compact 3-manifold $\bar M$. 

A {\bf compact core} $C\subset M$ for $M$ is a compact submanifold such that the inclusion $C\hookrightarrow M$ induces an isomorphism on fundamental groups. The existence of compact cores for 3-manifolds with finitely generated fundamental groups is a central result in the study of $3$-manifolds due to Scott (\cite{scott-cc}). Since $M$ is tame, there is a compact core for $M$ homeomorphic to $\bar M$ so that each component of $M-C$ is homeomorphic to $S\times \reals$ where $S$ is a component of $\partial\bar M=\partial C$. We call such a component (or its closure) an {\bf end} of $M$. Although this definition depends on the choice of $C$, it is easy to see that given a compact set $K\subset M$, we can choose $C$ so that $K\subset C$. Thus the asymptotic behavior of the ends of $M$ does
 not depend on the choice of $C$.

Let $G$ be a Kleinian group. Its limit set $\Lambda_G$ is the closure in the boundary at infinity $\partial\bH^3$ of the orbit of a base point. More precisely, fix a base point $O\subset\bH^3$ and set $GO=\{gO, g\in G\}$, then $\Lambda_G=\bbar{GO}\cap\partial\bH^3$. The convex core of $M=\bH^3/G$ is the quotient ${\rm Hull}(\Lambda_G)/G$ of the convex hull in $\bH^3$ of the limit set. Equivalently it is the smallest convex subset of $M$ whose inclusion
induces a homotopy equivalence with $M$.

Let us now assume that $G$ is finitely generated, has no parabolic elements and is not a lattice (i.e.\ $M=\bH^3/G$ has infinite volume). An end of $M$ is {\bf degenerate} if it lies in its convex core. The manifold $M$ is degenerate if all its end are degenerate, equivalently its convex core is the whole manifold, equivalently $\Lambda_G=\partial\bH^3$.

Work of
Thurston, Bonahon and Canary (\cite{thurstonnotes}, \cite{bonahon-bouts} and \cite{canary})
along with tameness \cite{agol-tameness, calegari-gabai}
 shows that degenerate ends are geometrically tame, i.e.\ there is a sequence of hyperbolic surfaces leaving every compact set:

\begin{theorem}
Let $E\approx S\times \reals$ be a degenerate end of a tame hyperbolic $3$-manifold, then there is a sequence of maps $f_n:S\to E$ such that 
\begin{enumerate}
	\item $f_n$ is homotopic to the map induced by the inclusion $S\times\{1\}\subset S\times \reals$,
	\item the metric induced on $S$ by $f_n$ is hyperbolic
	\item $f_n(S)\subset S\times [n,\infty)$ for any $n\in\natls$.
\end{enumerate}
 
\end{theorem}

This lead to the the definition of an ending lamination associated to a degenerate end. Consider a sequence of simple closed curves $c_n\subset S$ such that $\ell_{f_n}(c_n)$ is a bounded sequence, where $\ell_{f_n}$ is the length associated to the metric induced on $S$ by $f_n$. Extract a subsequence such that $c_n$ converges to a projective measured geodesic lamination $\lambda$ on $S$ (see \cite[Chap. 8]{thurstonnotes} and \cite{casson-bleiler} for definitions and properties of measured geodesic laminations). It follows from \cite{bonahon-bouts} that the geodesic lamination $\nu$ supporting $\lambda$ does not depend on the choices of $\{f_n\}$ or $\{c_n\}$, $\nu$ is the {\bf ending lamination of $E$}.

Notice that the induced metric on $\partial E$ is not hyperbolic. We will need to define geodesic laminations on $\partial E$. For this purpose, we fix a reference hyperbolic metric on $S$. Then we have a bilipschitz homeomorphism between $S$ and $\partial E$ endowed with the induced metric and a geodesic lamination on $\partial E$ is simply defined as the image of a geodesic lamination on $S$. 

\subsection{Cannon-Thurston maps, Exiting geodesics and limit set}

\begin{defn} A geodesic ray in $\til M$ 
is {\bf exiting} if it is a lift of an exiting geodesic ray in $M$, i.e.\ a lift of a geodesic ray
that is properly embedded in an end of $M$.
Let $E$ be an end of $M$ with $S=\partial E$. A geodesic  ray $\gamma$ in $M$ is {\bf exiting in $E$} if $\gamma$ is properly embedded and $\exists T$ such that $\gamma([T,\infty))\subset E$. Notice that exiting geodesic must be exiting in some end.
A {\bf minimizing geodesic segment} $\gamma$ through $p \in E$ is a geodesic segment 
between some $o \in S$ and $p$ with length equal to
$d_M (S,p)$. A lift of a minimizing geodesic segment to $\til M$ is called a minimizing geodesic segment in $\til M$.
\end{defn}

\begin{defn}   Let $X$ and $Y$ be hyperbolic metric spaces and
$i : Y \rightarrow X$ be an embedding. Suppose that a continuous extension $\hat{i}:
\widehat{Y} \to \widehat{X} $ of $i$ exists between their (Gromov) compactifications.
Then the boundary value of $\hat i$, namely
 $\partial\hat i : \partial Y \to \partial X$ is called
the {\bf Cannon-Thurston map}.
\end{defn}

Sometimes, in the literature \cite{mitra-trees, mitra-ct}, $\hat i$ is itself called a Cannon-Thurston map.
For us, $Y$ will be a Cayley graph of $\Ga$. We will be particularly interested in the case
that $\Ga$ is a surface Kleinian group isomorphic to $\pi_1(S)$ 
(for $S$ a closed surface of genus $g>1$). Also $X$ will be $\Hyp^3$, where we identify (the
vertex set of) $Y$ with an orbit of $\pi_1(S)$ in $\Hyp^3$. Equivalently (as is often done in
geometric group theory) we can identify $Y$ with $\Hyp^2
= \til{S}$, $X$ with $\Hyp^3$, and $i$ with the lift to universal covers
of the inclusion of $S$ into an end $E$ of $M$.
Then the main Theorems of \cite{mahan-split, mahan-red, mahan-elct, mahan-kl} gives us:

\begin{theorem} Let $S$ be the boundary of a degenerate  incompressible end $E$.
A Cannon-Thurston map $\partial i$ exists for $i: \til{ S} \to \til{E}$.
 Let $\LL_E$ denote the ending lamination corresponding to $E$.
Then $\partial i$ identifies $a, b \in \partial \til{ S}$ 
 iff $a, b$ are  end-points of a leaf of $\LL_E$ or boundary points of an ideal polygon whose sides are leaves
of $\LL_E$.

More generally for a degenerate $M$ without parabolics, let $K$ denote a compact core. Identify the
Gromov boundary $\partial
\til{K}$ with $\partial \Gamma$. Then a 
Cannon-Thurston map $\partial i$ exists for $i: \til{ K} \to \til{M}$. Also, 
$\partial i$ identifies $a, b \in \partial \til{K}$ 
 iff $a, b$ are  end-points of a leaf of $\LL_E$ or boundary points of an ideal polygon whose sides are leaves
of $\LL_E$ for some (lift of) an ending lamination $\LL_E$ corresponding to an end $E$.
 \label{ctstr}
\end{theorem}

Given Theorem \ref{ctstr} we define

\begin{defn} The {\bf multiple limit set}
 $\Lambda_m = \{ x \in \Lambda : \# |(\partial i)^{-1}(x)| > 1\}$. Equivalently

\begin{center}

$\Lambda_m = \{ \partial i (y) : y$ is an end-point of a leaf of   $\LL_E$ for some ending lamination $\LL_E \}$

\end{center}

\end{defn}

An infinite geodesic ray $[o,x) \subset \til{M}$ (where $x \in \partial {\til{M}}$) is said to {\bf land} at $x$.

A consequence of the construction in Section 4.2 of \cite{mahan-elct} is:

\begin{prop} \label{multexits} If $x \in \Lambda_m$, then $[o,x)$ is exiting.\end{prop}

In Proposition \ref{multi-almost-minimizing}, we will prove that $[o,x)$ is almost minimizing which is a stronger conclusion. Thus Proposition \ref{multexits} will also follow from Proposition \ref{multi-almost-minimizing};
and we omit the proof for now.

\subsection{Relationships between limit sets} We start with the observation that the conical limit set is contained in the
complement of the multiple limit set.

\begin{lemma}  $\Lambda_c \subset \Lambda_m^c$. \label{jklo} \end{lemma}

In a
 general form, this has been proven by Jeon, Kapovich, Leininger and Ohshika \cite{jklo}.
We shall give a different proof specialized to our context.

\begin{proof} Recall that a point $z \in \Lambda$ is conical if, given a base point $o\in \Hyp^3$
 there exists $R > 0$ such that there exist infinitely many $g \in \Ga$ satisfying
$g.o \in N_R ([o,z))$. Equivalently, $z$ is conical if and only if $[o,z) \cap N_R(\til{K})$ is
unbounded, i.e.\
 $[o,z)$  visits $N_R(\til{K})$ infinitely often.

Hence  $z$ is
non-conical if and only if for all $R>0$, $[o,z) \cap N_R(\til{K}) $ is 
bounded, i.e.\ $[o,z)$ is  exiting.
In particular, since  $z \in \Lambda_m$ implies that $[o,z)$ is exiting (by Proposition
\ref{multexits}), it follows that $z$ is
non-conical. The Lemma follows.
\end{proof}

\begin{prop} $\Lambda_c$ is a proper subset of $\Lambda_H$. \label{hnotc}
\end{prop}

\begin{proof} $[o,z)$ lands in the complement of $\Lambda_c$ if and only if it is exiting. On the other hand, 
$[o,z)$ lands in the complement of $\Lambda_H$ if and only if it is almost minimizing by Proposition \ref{ameqnt}.

It therefore suffices to find exiting rays that are not almost minimizing. Let $[o,z)$
be an almost minimizing ray in $E$. To construct $[o,z)$ just take a sequence $z_n$ exiting $E$, join them to $S (= \partial E)$ by minimizing geodesic segments, and take a limit.

Next, choose $w_i \in [o,z)$ such that $d(w_i, w_{i+1}) > D_0$, for a large $D_0$ (to be fixed later).
Let $\sigma_i$ be closed geodesics of length at most 1 such that 
\begin{enumerate} 
\item If the length of $\sigma_i$ is larger than the Margulis constant, then  $d(w_i, \sigma_i) \leq 1$. 
\item Else, if $T_i$ is the Margulis tube containing $\sigma_i$, then  $d(w_i, T_i) \leq 1$.
\end{enumerate}
 We let $\alpha_i$ be a loop based at $w_i$ winding $n_i$ times around $\sigma_i$ obtained by
adjoining initial and final
segments of length at most one joining $w_i$ to $\sigma_i$ or $T_i$, and winding $n_i$ times around $\sigma_i$
in between. 

Then the concatenation $\bigcup_i ([w_{i-1}, w_i] \cup \alpha_i)$ is exiting in $E$
and lifts to a quasigeodesic $\eta$ with bounded constants
(provided $D_0$ and
$n_i$'s are sufficiently large) in $\til E$. The geodesic tracking $\eta$ is then an example
of an exiting geodesic that is not almost minimizing.
\end{proof}

Our last goal in this section is to relate the horospherical and multiple limit sets.
We shall need the following consequence of Thurston's \cite[Chapter 9]{thurstonnotes} result that
the ending lamination corresponding to a degenerate end is well-defined.

\begin{lemma} \label{endpts}
Let $S = \partial E$ be the boundary of a degenerate end $E$
with ending lamination $\LL_E$. Assume that $S$ is equipped with a hyperbolic structure.
Let $\alpha_n$ be a sequence of closed geodesics in $S$ whose geodesic realizations $\sigma_n$ exit $E$.
Denote by $\alpha_n^{\pm \infty}$ the attracting and repelling fixed points of $\alpha$ on
$\partial {\til{S}} = S^1$. If $p$ is a limit (in $S^1$) of a (subsequence of) 
$\alpha_n^{\infty}$, then $p$ is the end-point of a leaf of $\LL_E$.

Conversely, any end-point of a leaf of $\LL_E$ is a limit of a (subsequence of) 
$\alpha_n^{\infty}$'s.
\end{lemma}

\begin{proof} Let $[\alpha_n] (\subset {\til{S}}) $ denote the bi-infinite geodesic
with end-points $\alpha_n^{\pm \infty}$. It follows from \cite[Chapter 9]{thurstonnotes} that any
subsequential limit (in the Gromov-Hausdorff topology) of $[\alpha_n]$'s is either a leaf of $\LL_E$
or contained (as a diagonal) in an ideal polygon in the complement of $\LL_E$. One direction of
the Lemma follows.

Further, since the Hausdorff limit of $\alpha_n$'s on $S$ contains $\LL_E$, the converse follows. 
\end{proof}

We shall now show:
\begin{prop}	\label{multi-almost-minimizing}
$z\in\Lambda_m$ if and only if $[o,z)$ is almost minimizing.
\end{prop}

\begin{proof}[Proof of Proposition \ref{multi-almost-minimizing}]
We continue with the notation of Lemma \ref{endpts}.
We first show that if $\overline r$
is the lift of an almost minimizing geodesic to $\til E$, then its end-point (in $\partial {\Hyp}^3$)
belongs to $\Lambda_m$. 

Let $r$ be an almost minimizing geodesic ray in $E$.
Then there exists $C > 0$ such that for all $t \in [0, \infty )$, there is a  closed unknotted essential
loop $\sigma_t$ ($\sigma_t$ is homotopic to a simple closed curve on $S$) such that $l(\sigma_t) \leq C$
and $d(\sigma_t, r(t)) \leq C$ (here we cannot assume that $\sigma_t$ is a geodesic, the geodesic realization may be very short, enclosed in a deep Margulis tube, in which case the distance to the geodesic realization
is quite large). 

Let $\alpha_t$ be a simple closed curve on $S=\partial E$ freely homotopic to $\sigma_t$.
Then there exists a geodesic segment $r_t$ of length in the interval $[t-2C, t+ 2C]$ joining 
$\alpha_t, \sigma_t$ such that the Hausdorff limit (as $t \to \infty$) of $r_t$ is asymptotic to $r$.
Let $\bbar{\alpha_t}, \bbar{\sigma_t}$ be geodesic segments in $\til E$
that are lifts of $\alpha_t, \sigma_t$, such that their initial
points
and end-points are connected by lifts  $r_{t1}, r_{t2}$ of $r_t$. (See diagram below)
Assume further that the initial point of $\bbar{\alpha_t}$ (and $r_{t1}$) is a fixed 
point $o \in \til E$ (independent of $t$), and let the end-point of $\bbar{\alpha_t}$ be denoted by $w_t$.


\begin{figure}[hbtp]
\psfrag{s}{$\bar\sigma_t$}
\psfrag{a}{$\bar\alpha_t$}
\psfrag{t}{$r_{1t}$}
\psfrag{r}{$r_{2t}$}
\psfrag{o}{$o$}
\psfrag{w}{$w_t$}

\includegraphics[height=4cm]{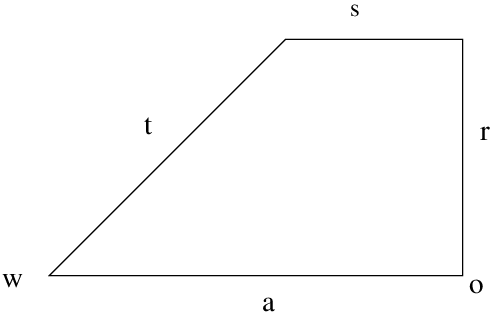}

\end{figure}

 By Lemma \ref{endpts}, $w_t$ converges (up to subsequence) to a point $z $ which is the end-point of a leaf 
of $\LL_E$.  Hence by Theorem \ref{ctstr}, $\hat{i} (z ) \in \Lambda_m$ (recall that
 $\hat i$ denotes the Cannon-Thurston
map).  

The concatenation $\gamma_t$ of
$r_{t1}$,  $\bbar{\sigma_t}$ and $r_{t2}$ (with orientation reversed)
is  a uniform  quasigeodesic. The proof of this statement is a replica of the argument
occurring in Lemma 3.5 of \cite{mitra-endlam}, Proposition 3.1 of \cite{mahan-elct} or Proposition 5.2 of \cite{mahan-kl}. In the last reference
a detailed proof is given and we omit the proof here. 

Let $r (\infty)$ denote the terminal
point of $r$ in $S^2 = \partial {\Hyp}^3$. Then, since $\gamma_t$ is a uniform quasigeodesic
containing $r_{t1}$ as an initial segment, it follows
that the end-points of $r_{t1}$ and $\gamma_t$ converge to the same point on  $S^2$, or equivalently,
$r (\infty) = \hat{i} (z )$. Hence $r (\infty) \in \Lambda_m$. Since $r$ was arbitrary, we have shown
that $ \Lambda_H^c \subset \Lambda_m$.\\

To prove the reverse inclusion, given $z \in \Lambda_m$ we will construct a geodesic ray $\tilde\gamma\subset\tilde M$ with endpoint $z$ (but whose initial point may not be $o$) whose projection to $M$ is minimizing. It will then follow from Lemma \ref{minimizing-almost-minimizing} that $[o,z)$ is almost minimizing.

Let $(a, b)$ be a bi-infinite leaf of $\LL_E$ such that $\hat{i} (a) = \hat{i} (b) =z$ (notice that there is always such a leaf according to \cite{mahan-elct} and \cite{mahan-kl}). Let $\alpha_n$ be a sequence of closed geodesics on $S= \partial E$ and $\tilde\alpha_n\subset\tilde S$ a leaf in the preimage of $\alpha_n$ such that 
\begin{enumerate}
\item $\tilde\alpha_n^{\pm \infty}$ converges to $\{ a, b \}$. 
\item The geodesic realizations $\sigma_n$ in $M$, of $\alpha_n$ in $E$, exit $E$.
\end{enumerate}

Let $\tilde\sigma_n$ be the leaf of the preimage of $\sigma_n$ with the same endpoints as $\tilde\alpha_n$. Let $\tilde x_n$ be a point of $\tilde\sigma_n$ and $\tilde z_n\subset\partial\tilde E$ be a point realizing the distance between $\tilde x_n$ and $\partial\tilde E$. The existence of $\tilde z_n$ comes from the properness of the embedding $\partial\tilde E\subset\tilde E$. If $\tilde z_n$ stays in a compact set, then, up to extracting a subsequence, the geodesic segments joining $\tilde z_n$ to $\tilde x_n$ converge to a geodesic ray $\tilde \gamma$ joining the limit $\tilde z_\infty$ of $\tilde z_n$ to $z$. Since $\tilde z_n$ realizes the distance between $\tilde x_n$ and $\partial\tilde E$, the projection $\gamma$ of $\tilde \gamma$ is a minimizing geodesic ray (since, as usual, a limit of minimizing geodesic segments
is a minimizing geodesic ray) and we are done.

Otherwise we pick a sequence $g_n\in\pi_1(S)$ such that $g_n\tilde z_n$ stays in a compact set. We are going to show that $g_n\tilde x_n$ also tends to $z$. Then up to extracting a subsequence, the geodesic segments joining $g_n\tilde z_n$ to $g_n\tilde x_n$ converges to a geodesic ray $\tilde \gamma$ joining the limit $\tilde z_\infty$ of $g_n\tilde z_n$ to $z$. Again, the projection $\gamma$ of $\tilde \gamma$ is a minimizing geodesic ray
(since a limit of minimizing geodesic segments
is a minimizing geodesic ray) and we are done.

We first show that $\tilde z_n$ can only exit toward $z$. Using the continuity of Cannon-Thurston map, we will show that this imposes some restrictions on the behavior of $g_n$ that will lead us to the expected conclusion ($g_n\tilde x_n$ tends to $z$).

\begin{claim}
Up to extracting a subsequence, either $\tilde z_n$ stays in a compact set or it tends to $z$.
\end{claim}

\begin{proof}
Assume that $\tilde z_n$ does not stay in a compact set. Then, up to extracting a subsequence, it converges to a point $\chi\in\Lambda_G$. Seeking a contradiction, assume that $\chi\neq z$. Then the geodesic segments $[\tilde x_n,\tilde z_n]$ converge to the geodesic $l$ with endpoints $(\chi,z)$. Pick $2$ points $\tilde y\subset l$ and $y'\subset\partial\tilde E$. For $n$ large enough we have $d(\tilde y, \tilde z_n)\geq d(\tilde y,\tilde y')+2$. Since $[\tilde x_n,\tilde z_n]$ converge to $l$, for $n$ large enough it passes nearby $\tilde y$ so that $d(\tilde x_n,\tilde z_n)\geq d(\tilde x_n,\tilde y)+d(\tilde y,\tilde z_n)-1\geq d(\tilde x_n,\tilde y)+d(\tilde y,\tilde y')+1$. This would contradict the assumption that $\tilde z_n$ realizes the distance between $\tilde x_n$ and $\partial\tilde E$.
\end{proof}

So let us assume that $\tilde z_n$ tends to $z$, in $\partial\tilde E$ it tends to a point $c$ in the ideal boundary such that $i(c)=z$.

Pick a sequence $g_n\in\pi_1(S)$ such that $g_n\tilde z_n$ stays in a compact set and denote by $(d_n,c_n)$ the attracting and repulsing points of $g_n$. Since $\tilde z_n$ tends to $c$, $c_n$ converge to $c$. Extract a subsequence such that $d_n$ converges to $d$.
If $d=c$, then $g_n(a)$ and $g_n(b)$ converge to $c$. It follows that $g_n\tilde\sigma_n$ tends to $i(c)=z$, in particular $g_n\tilde x_n$ tends to $z$.
If $d\neq b$, the axis of $g_n$ tends to the geodesic with endpoints $(c,d)$. In particular the distance from $g_n\tilde z_n$ and hence from $\tilde z_n$ to the axis of $g_n$ is bounded. It follows that the distance from $\tilde z_n$ to $\tilde\alpha_n$ is bounded. Then since $g_n\tilde z_n$ stays in a compact set, up to extracting a subsequence, $g_n\tilde\alpha_n$ converges to a leaf of the preimage of $\LL_E$. This is only possible if $(a_n,b_n)$ and $(g_na_n,g_nb_n)$ converge to $(c,d)$ or $(d,c)$. It follows that $g_n\tilde x_n$ tends to $z$.

\end{proof}
We should remark here that the construction of an almost minimizing geodesic in the above proof can very well furnish
minimizing ones. We refer the reader to Remark \ref{rem-minam} for a clarification on why we have decided
to deal with almost minimizing rather than  minimizing geodesics.
Combining Propositions \ref{ameqnt} and \ref{multi-almost-minimizing}, we get:

\begin{cor} \label{horo-mult}
$\Lambda_H^c = \Lambda_m$. 
\end{cor}

\begin{cor}
For each degenerate end $E$ with incompressible (in $M$)
 boundary $S$,  there is an $\R-$tree $\TT_E (\subset \Lambda)$ dual to $\LL_E$ parametrizing 
the lifts of almost minimizing geodesics exiting $E$. Hence $\Lambda_H^c$ is a disjoint union of
$\TT_E$'s -- one for each lift of $E$ as $E$ ranges over the degenerate ends of $M$.
\end{cor}

\begin{proof} By Proposition \ref{horo-mult}, it suffices to obtain a description of $\Lambda_m$. 
Also, by Proposition \ref{horo-mult}, and Theorem \ref{ctstr}, $\Lambda_m$ is the set of equivalence classes
in $\LL_E$, where $a, b \in S^1$ are equivalent if they are end-points of a leaf of $\LL_E$
or ideal end-points of a complementary ideal polygon. By joining all such 
$a, b$ by bi-infinite geodesics, and collapsing
leaves and ideal complementary polygons down to points, it follows
that $\Lambda_m$ is the dual $\R-$tree $\TT_E$ to $\LL_E$. The last statement is an immediate consequence.
\end{proof}

\section{Building Blocks and Model Geometries}\label{models}

Having discussed which geodesics are almost-minimizing as opposed to merely exiting in Section \ref{hll},
it
remains to discuss conditions guaranteeing thickness or thinness of
 almost minimizing geodesics in $E$.
 We have already seen examples in Subsections \ref{thick} and \ref{thin}
where all almost minimizing geodesics are thin and an example
where all almost minimizing geodesics are thick.  The purpose of the rest of this paper is to furnish
conditions in special cases and explore the limitations of these conditions. As the examples of 
subsection \ref{thick} and \ref{thin} indicate, the geometry of building blocks plays a crucial role.
To proceed further,
 we pick up model geometries of ends of hyperbolic 3-manifolds following \cite{minsky-bddgeom, minsky-elc1, minsky-elc2, mahan-bddgeo, mahan-ibdd, mahan-amalgeo, mahan-split} one by one and discuss the behavior
of almost minimizing geodesics
 for each.

In what follows in this section we shall describe different kinds of models for building blocks
of $E$: thick, thin, amalgamated.
Each building block is homeomorphic to $S \times [0,1]$, where $S$ is a closed
surface of genus greater than one. What is common to all these three model building blocks
is that the top and bottom boundary components are uniformly bi-Lipschitz to a fixed hyperbolic $S$.
In the next section, a more general model geometry will be described and almost minimizing geodesics in it 
will be analyzed.
\begin{defn} 
A model $E_m$ is said to be built up of blocks of some prescribed geometries {\bf glued end to end}, if 
\begin{enumerate}
\item 
$E_m $ is homeomorphic to $S \times [0, \infty)$
\item There exists $L \geq 1$ such that $S \times [i, i+1]$ is  $L-$bilipschitz to a block of one of the prescribed
geometries
\end{enumerate} 
$S \times [i, i+1]$ will be called the {\bf $(i+1)$th block} of the model $E_m$.

The {\bf thickness} of the \bf $(i+1)$th block is the length of the shortest path between $S \times \{ i \}$
and $S \times \{ i+1 \}$ in $S \times [i, i+1] (\subset E_m)$.
\end{defn}
 
\subsection{Bounded Geometry} 

\begin{defn}
An end $E$ of a hyperbolic $M$ has {\bf bounded geometry} 
\cite{minsky-bddgeom, minsky-jams} if there is a (uniform)
lower bound for lengths of closed geodesics in $E$.
\end{defn}

 Since $E$ itself has bounded geometry, it follows that any
exiting geodesic is thick. We note this as follows for future use:

\begin{lemma} Let $E$ be of bounded geometry. Then every exiting geodesic is thick. In particular
every almost minimizing geodesic is thick. \label{bddgeo} \end{lemma}

A bi-Lipschitz model $E_m$ for $E$ may be described by gluing a sequence of {\bf thick blocks}
end-to-end. We describe below the construction of a thick block, as this will be used in all the model geometries
that follow.

 Let $S$ be a closed hyperbolic surface with a fixed but arbitrary
hyperbolic structure. 

\medskip

\noindent 
{\bf Thick Block}\\ 
Fix a constant $L$ and a hyperbolic surface $S$. Let $B_0 = S \times [0,1]$ be given the product metric.
 If $B$ is $L-$bilipschitz homeomorphic to $B_0$, it is called an {\bf $L-$thick  block}.

The following statement is a consequence of 
\cite{minsky-jams} (see also \cite{mitra-trees, mahan-bddgeo}).

\begin{remark} \label{bddgeo-model} For any bounded geometry end, there exists $L$ such that
$E$ is bi-Lipschitz homeomorphic to a model manifold $E_m$ 
consisting of gluing  $L-$thick blocks end-to-end.
\end{remark}

Later on in this section, we shall omit stating the constant $L$ explicitly, but assume that, given an end $E$,
 this constant
is uniform for thick blocks in $E$.

\subsection{i-bounded Geometry} \label{ibdd}

\begin{defn}
An end $E$ of a hyperbolic $M$ has {\bf i-bounded geometry} \cite{mahan-ibdd}
if the boundary torus of every Margulis tube in $E$ has bounded diameter. 
\end{defn}

An alternate description of i-bounded
geometry can be given as follows.
We start with a closed hyperbolic surface $S$.
  Fix a finite collection $\CC$ of disjoint simple closed geodesics on
$S$ and let
 $N_\epsilon ( \sigma_i )$  denote an $\epsilon$ neighborhood of
$\sigma_i$, ($\sigma_i \in \mathcal{C}$). Here $\epsilon$ is chosen
 small enough so
that no two lifts of  $N_\epsilon ( \sigma_i )$ to the
universal cover $\widetilde{S}$ intersect.

\medskip

\noindent {\bf  Thin  Block}\\
Let $I = [0,3]$. Equip $S \times I$ with the product metric. Let
$B^c = (S \times I - \cup_i N_{\epsilon} ( \sigma_i ) \times
[1,2]$. 
Equip $B^c$ with the induced path-metric.

For each resultant torus component of
the boundary of $B^c$, perform Dehn filling on some $(1,n_i)$ curve (the $n_i$'s may vary from
block to block but we do not add on the suffix for $B$ to avoid cluttering notation),
which goes $n_i$ times around the meridian and once round the
longitude.  $n_i$ will be called the {\bf twist coefficient}. Foliate the torus boundary of
$B^c$ by translates of $(1,n_i)$ curves and arrange so that the 
 solid torus $\Theta_i$ thus glued in is hyperbolic and foliated by totally geodesic 
disks bounding the $(1,n_i)$ curves.
 $\Theta_i$ equipped with this metric will be called a {\bf
  Margulis tube}.

The resulting copy of $S \times I$ obtained, equipped with the metric just
described, is called a {\bf thin  block}. 
The following statement is a consequence of 
 \cite{mahan-ibdd}.

\begin{prop}\label{ibdd-model}
An end $E$ of a hyperbolic 3-manifold $M$ has
 {\bf i-bounded geometry} if and only if
it is bi-Lipschitz homeomorphic to a model manifold $E_m$ 
consisting of gluing  thick and thin blocks end-to-end.
\end{prop}

The figure below illustrates a model $E_m$, where
the black squares denote Margulis tubes and the (long) rectangles without black squares
represent thick blocks.

\smallskip

\begin{center}

\includegraphics[height=4cm]{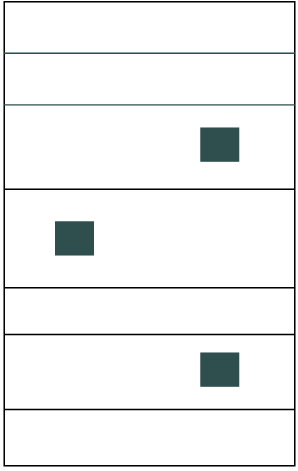}

\smallskip

 {\it Model of i-bounded geometry (schematic)}

\end{center}

\smallskip

\begin{prop} Let $E$ be of i-bounded geometry. Then there exist thin exiting geodesics. However,
every almost minimizing 
geodesic is thick. \label{ibddgeo} \end{prop}

\begin{proof}
{\bf Existence of thin exiting geodesics:} The proof of this is similar to Proposition \ref{hnotc}.
 Let $B_{n_i}$ be the thin blocks
with $n_i < n_{i+1}$. 
Let $T_{n_i}$ and $T_{n_{i+1}}$ be  Margulis tubes in these blocks.
 We choose thick minimizing geodesics $\lambda_i$ between $T_{n_i}$ and $T_{n_{i+1}}$, so that the length of 
 $\lambda_i$ is given by $l(\lambda_i)= d(T_{n_i}, T_{n_{i+1}})$.

Now consider long geodesic paths $\mu_i \subset T_{n_i}$ in Margulis tubes winding $m_i$
times around the core of $ T_{n_i}$, where $l(\mu_i) \to \infty$
as $i \to \infty$. The rest of the proof is as in 
Proposition \ref{hnotc}.  The concatenation $\bigcup_i (\lambda_i\cup \mu_i)$ is exiting
and lifts to a (uniform) exiting quasigeodesic
which is thin. Hence the exiting geodesic that tracks it is thin.\\

\noindent 
{\bf Thickness  of almost minimizing geodesics:}  Let $\lambda$ be an almost minimizing geodesic
and $\lambda_i' = \lambda \cap B_i$ be the piece(s) of it within the block $B_i$. Let $\lambda_i$
be the geodesic subsegment of $\lambda$ between the first intersection point of $\lambda$
with $S \times \{ i \}$ and the last intersection point of $\lambda$
with $S \times \{ i \}$. Since there exists
$L \geq 1$ such that each $S \times \{i\}$
is $L-$bilipschitz homeomorphic to $S$, it follows that there exists $L_0$ such that
for  any $p_i \in S \times \{i\}$
and
$p_{i+1} \in S \times \{i\}$, there exists a path of length at most $L_0$ joining $p_i$ to 
$p_{i+1}$. To see this, choose $x \in S_i \setminus (\bigcup_j N_{\epsilon} ( \sigma_{i_j} ))$, where the  
$\sigma_{i_j}$'s correspond to the Margulis tubes in $B_{i+1}$. Then $x \times I$ is a thick path
of length at most $3L$, where $x \times \{ 0 \}$ corresponds to $x \in S_i$ and 
$x \times \{ 3 \}$ lies on $S_{i+1}$. Since  $S_i$'s are of bounded geometry, i.e.\ there exists $D > 0$
such that the diameter of $S_i$ is bounded by $D$ for all $i$, it follows that there is a path between
$p_i$ and $p_{i+1}$ of length at most $(2D + 3L)$. We choose $L_0 = (2D + 3L)$.

Since $\lambda$ is almost minimizing, there exists $C \geq 0$ such that $l(\lambda_i ) \leq (L_0 +C)$
for all $i$.  In particular, $\lambda_i$ cannot go arbitrarily deep into the Margulis tube $T_i$
in case $B_i$ is thin. It follows that $\lambda$ is thick.
\end{proof}

\subsection{Amalgamation Geometry}
As before, we
start with a closed hyperbolic surface $S$. An amalgamated geometry block is similar to a thin block, except that
we impose no control on the geometry of $S\times [1,2] \setminus (\bigcup_j N_{\epsilon} ( \sigma_{i_j} )\times
[1,2])$.

\begin{defn} {\bf Amalgamated  Block}
{\rm
As before $I = [0,3]$. We will describe a geometry on $S \times I$. 
 There
exist $\epsilon, L$ (these constants will be uniform over blocks of the model $E_m$)
such that 
\begin{enumerate}
\item $B = S \times I$. Let $K = S \times [1,2]$ under the identification of $B$ with 
$S \times I$.
\item We call $K$ the geometric core.
In its intrinsic path metric, it is  $L-$bilipschitz to a 
  convex hyperbolic manifold with boundary consisting of surfaces L-bilipschitz to a fixed hyperbolic
surface. It follows that there exists $D>0$ such that
 the diameter of $S \times \{
  i \}$ is bounded above by $D$ (for $i = 1,2$). 
\item There exists a simple closed multicurve  on $S$ each component of which has a geodesic realization
 on $S \times \{ i
  \} $ for  $ i = \{ 1,2\}$ with (total) length at most
 $\epsilon_0$. Let $\gamma$ denote its geodesic realization in $K$.
\item There exists a regular neighborhood $N_k ( \gamma ) \subset K$
  of $\gamma$ which is homeomorphic to a union of disjoint solid tori, such that $N_k
  (\gamma ) \cap S \times \{ i \}$ is homeomorphic to a union of disjoint open annuli
  for $i = 1, 2$. Denote $N_k (\gamma )$ by
  $T_\gamma$ and call it the Margulis tube(s) corresponding to $\gamma$.  
\item $S \times [0,1]$ and $S \times [2,3]$ are given the product
  structures corresponding to the bounded geometry structures on $S
  \times \{ i \}$, for $i = 1,2 $ respectively. 
\end{enumerate} }
\end{defn}

In \cite{mahan-amalgeo} we had imposed further restrictions on the geometry of the geometric core $K$.
But for the purposes of this paper, the above is enough.

\begin{defn}\label{amalgeo-model}
An end $E$ of a hyperbolic 3-manifold $M$ has
 {\bf amalgamated geometry} if 
it is bi-Lipschitz homeomorphic to a model manifold $E_m$ 
consisting of gluing  thick and amalgamated blocks end-to-end.
\end{defn}

\begin{remark} Note that with the above Definition, i-bounded geometry becomes a special case of amalgamated geometry.
The difference is that amalgamated geometry imposes no conditions on the geometry of the complement $K - T_\gamma$.
The component of $K - T_\gamma$ shall be called {\bf amalgamation components} of $K$.
\end{remark}

The figure below illustrates schematically what the model looks
like. Filled squares correspond to solid tori along which amalgamation
occurs. The adjoining piece(s) denote amalgamation blocks of $K$.
 The blocks which have no filled squares are the
{\it thick blocks} and those with filled squares are the {\it amalgamated
  blocks}

\smallskip

\begin{center}

\includegraphics[height=4cm]{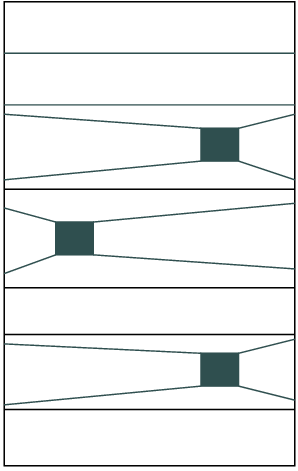}

\smallskip

 {\it Model of amalgamated geometry (schematic)}

\end{center}

\subsubsection{Almost minimizing geodesics} 
Recall that the thickness of the $(i+1)$th block is the length of the shortest path between $S \times \{i\}$ and
$S \times \{i+1\}$ in $S \times [i,i+1]$.
\begin{lemma}\label{amalgeo-suff}
Any almost minimizing geodesic in an amalgamated geometry end is thick if  all amalgamated
blocks have bounded thickness. 
\end{lemma}

\begin{proof} 
The proof is similar to the second part of Proposition \ref{ibdd-model}.

Let $\lambda$ be an almost minimizing geodesic
and $\lambda_i$
be the geodesic subsegment of $\lambda$ between the first intersection point of $\lambda$
with $S \times \{ i \}$ and the last intersection point of $\lambda$
with $S \times \{ i +1 \}$. If each $B_i$ has thickness bounded by $D_0$ then
as in the proof of Proposition \ref{ibdd-model},   there exists $L_0$ such that
for  any $p_i \in S \times \{i\}$
and
$p_{i+1} \in S \times \{i+1\}$ there is a path of length at most $L_0$ joining $p_i$ to 
$p_{i+1}$. To see this, note first that there is a path of length at most $D_0$ from 
$S \times \{ i \}$ to  $S \times \{ i +1 \}$. Further, since each $S \times \{ i \}$ has diameter bounded
by some $L$, $L_0 = D_0 +2L$  will suffice.

Since $\lambda$ is almost minimizing, there exists $C \geq 0$ such that $l(\lambda_i ) \leq (L_0 +C)$
for all $i$.  In particular, $\lambda_i$ cannot go arbitrarily deep into the Margulis tube $T_i$
in case $B_i$ is an amalgamated block. It follows that $\lambda$ is thick.
\end{proof}

\section{Counterexamples: Bounded thickness neither necessary nor sufficient}\label{ctreg} The converse to Lemma \ref{amalgeo-suff}   is not true. In particular, as we shall see 
in Example \ref{unbddth} below, 
it is possible to have all almost minimizing geodesics thick in manifolds of amalgamation
geometry even in the presence of arbitrarily thick
amalgamation blocks. Further,
for more general geometries of ends ((than amalgamated geometry), bounded thickness of blocks is not sufficient 
to guarantee that almost minimizing geodesics are thick.
In Section \ref{ctreg-split} we shall provide a counterexample.
Thus, for general geometries, bounded thickness of blocks is neither necessary nor sufficient to ensure thickness 
of almost minimizing geodesics.
For both these counterexamples, we shall need the more general technology of split geometry.
 This is summarized in the Appendix to the paper (Section \ref{splitt}).
For convenience of the
reader we shall refer to specific sections of the Appendix that are used.

\subsection{Example: Converse  to
Lemma \ref{amalgeo-suff} is false}\label{unbddth} We
 furnish here a counterexample to the converse direction to
Lemma \ref{amalgeo-suff} as follows. We shall proceed in two steps:
\begin{enumerate}
\item Construct an amalgamation geometry
block.
\item Describe how to glue a sequence of such blocks together.
\end{enumerate}

The gluing method will construct for us a hierarchy, which suffices (\cite{minsky-elc2}, see
Theorem \ref{bilipmodel} in the Appendix)
to furnish the example we seek.

\subsubsection{The Amalgamation Geometry Block}\label{agblock}
Instead of constructing a complete amalgamation geometry
block, we shall describe the construction only in an amalgamation component. See figure below.

\smallskip

\begin{center}

\includegraphics[height=4cm]{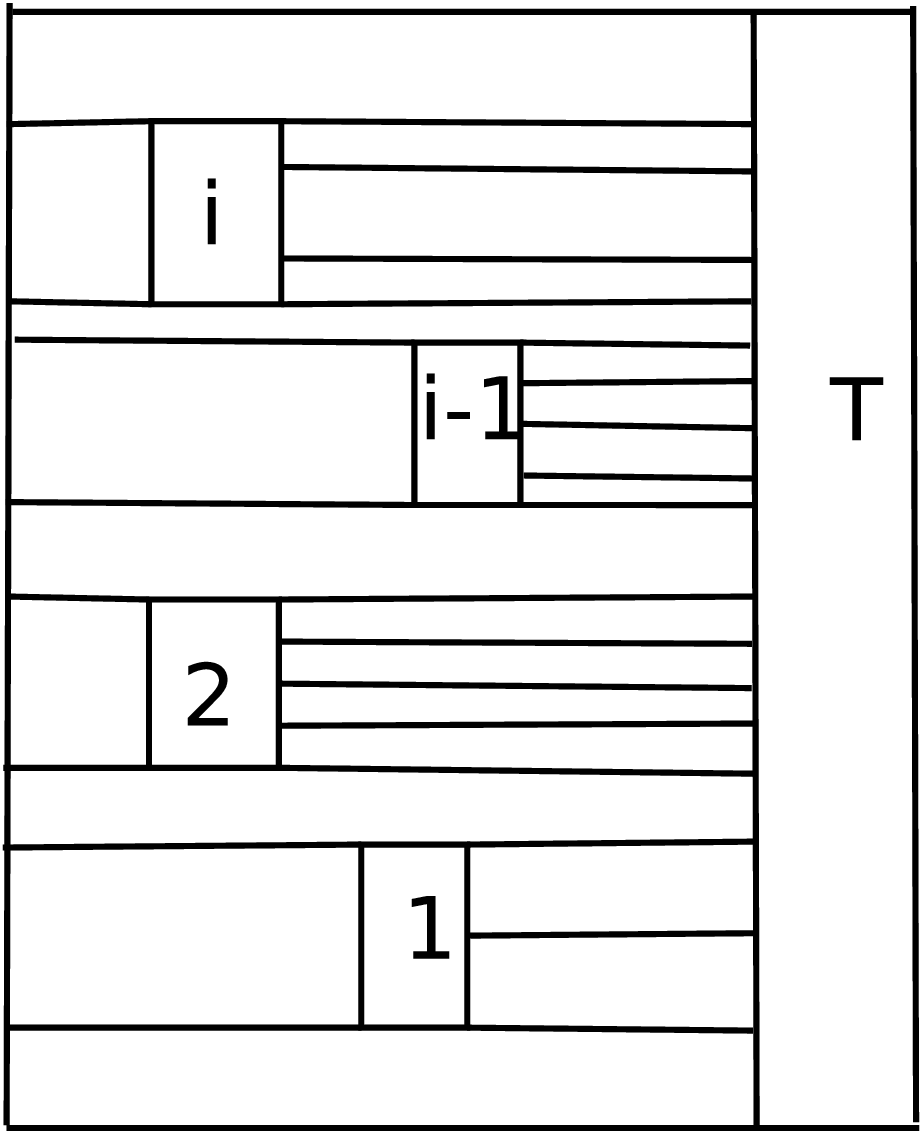}

\smallskip

\end{center}

\smallskip

Let $K$ be an amalgamation component (homeomorphic to $S_K \times I$)
 bounding a thin Margulis tube $T$ which begins and ends at bounded geometry 
surfaces. The left vertical boundary of
 $T$  has length $\sum_{j=1}^i (m_j + n_j)$ corresponding to Minsky blocks
(Section \ref{mb}) abutting on it. 
We shall say below what the $m_j, n_j$ are.

Let $v$ be the curve corresponding to $T$ on the surface $S$.
There are $  i (= i(K))$ hierarchy curves $v_1, v_2, \cdots , v_{i-1}, v_i$ corresponding to Margulis tubes
$T_1, T_2, \cdots , T_{i-1}, T_i$ (labeled $1,2, \cdots , i-1, i$ in the figure above) inside the amalgamation component.

$S_K \setminus \{ v_j \}$ has two components $W_j$ and $V_j$ which are component domains for hierarchy geodesics
(see Section \ref{hier}).
The component domain $W_j$ with $v, v_j$ as boundary components supports a thick hierarchy geodesic
 of length   $n_j$.  The other component domain $V_j$ (with only $v_j$ as its boundary,
lying to the left of $v_j$ in the picture) has a length one (or uniformly bounded length in general) hierarchy
geodesic segment supported on it.

Further, between the last split  surface containing $v, v_j$ and the first split  surface containing $v, v_{j+1}$
all split surfaces are thick and the component domain $S_K \setminus \{ v \}$ supports a thick 
hierarchy geodesic of length $m_j$.

 This forces the left vertical boundary of $T$ to have length $\sum_{j=1}^i (m_j + n_j)$. On the other hand
 the length of a path $\eta$ (say) in $K$ in the left part of the
picture  is of the order of $\sum_j (m_j +1)$.

 The shortest path through the Margulis tube $T$ has length of the order of
 $log(\sum_{j=1}^i (m_j + n_j))$
and by choosing $n_i$ large one can make the length of $\eta$ (i.e.\ $\sum_j (m_j +1)$)
the thickness of $K$. 

A similar construction is done to the right of $T$ with similar estimates, so as to get finally an amalgamated
block $B$ which has thickness of the order of $\sum_j (m_j +1)$.

\subsubsection{Gluing Blocks Together and producing a hierarchy}
We finally indicate how to glue several such blocks together to give a model for a hyperbolic manifold.
The hierarchy machinery comes quite handy here. Instead of blocks, we describe a hierarchy path,
the correspondence between these two descriptions being given by \cite{minsky-elc1, minsky-elc2}
as summarized in Theorem \ref{bilipmodel} and Definition \ref{wsplitrmk} in the Appendix.\\

Thus, we choose amalgamation geometry blocks $B_1, \cdots, B_l, \cdots$
and between the top boundary of $B_l$ and the bottom boundary of $B_{l+1}$ we 
glue a sequence of thick blocks. In the hierarchy, this corresponds to
a thick Teichmuller geodesic $\mu_l$ of length $d_l$, which remains thick in the curve complex $\CC(S)$.
Thus, successive $\mu_l, \mu_{l+1}$ are attached at a vertex $v_l$ (corresponding to the Margulis tube
$T$ in the above construction). However (a la \cite{DGO}), the incoming $\mu_l$ and the outgoing 
$\mu_{l+1}$ make an `angle' at  $v_l$ approximately equal in size to the length of the hierarchy path 
corresponding to $B_l$ in the link of $v_l$. Since this `angle' is large, the concatenation of
the $\mu_l$'s along with the hierarchy paths corresponding to $B_l$ is a quasigeodesic $r$ in $\CC(S)$
ensuring that the gluing of the blocks in order actually approximates the split geometry model
of  the hyperbolic 3-manifold whose ending lamination is given by $r(\infty) \in \partial CC(S) 
= \EL (S)$ (c.f.\ \cite{klarreich-el}).

A word about the hierarchy within the amalgamation geometry block is necessary here. The gluing pattern 
constructed in Section \ref{agblock} also constructs a quasigeodesic in the relevant component domain.
Hence the {\it qualitative} features of the constants used survive when we pass to the hierarchy determined by
the ending lamination. We note, in particular,  that even after passing to the actual hierarchy
 the {\it difference} between the lengths of the left vertical boundary and the right vertical boundary does not go
to infinity. This is because there is a bilipschitz map between the model built from the quasigeodesic approximation 
of the hierarchy (given by the gluing pattern) and that built from the actual hierarchy. Further there is
a uniform constant depending
only on that of the quasigeodesics. Hence, a posteriori (using Theorem \ref{bilipmodel}), the additive errors in passing from the
quasigeodesic approximation of  the hierarchy
to the actual hierarchy  are all uniformly bounded. \\ 

\medskip

\noindent {\bf Summary:}\\
We summarize the features of the above example:

\begin{enumerate}
\item The boundaries of the amalgamation blocks have (uniformly) bounded geometry
as required in a model
of amalgamated geometry.
\item the almost minimizing paths are thick
\item the amalgamation geometry blocks have unbounded thickness.
\end{enumerate}

Hence thick almost minimizing paths may exist even
when amalgamation geometry blocks have unbounded thickness. This  example
shows that the converse  to
Lemma \ref{amalgeo-suff} is false.

\subsection{Bounded thickness does not imply thick almost minimizing geodesics}\label{ctreg-split}
 In this subsection we give an example to show that even when all split blocks have bounded thickness, it is not necessary that almost minimizing geodesics be thick. Thus, the analog of Lemma \ref{amalgeo-suff} is false
in the general case of split geometry and hence by Theorem \ref{minsky-split}
 for degenerate hyperbolic 3-manifolds in general. As in Section \ref{unbddth}
we shall build an approximation to a hierarchy whose qualitative features pass to the genuine
hierarchy determined by the ending lamination corresponding to the base quasigeodesic of the approximate
hierarchy.

Here is the idea of the construction using the 
model of split geometry from Theorem \ref{bilipmodel} and Definition \ref{wsplitrmk} in the Appendix.
 We work with a sphere with $n$ holes for convenience. It is
straightforward to generalize the construction below to closed surfaces. The aim is to
first construct a split
geometry model such that

\begin{enumerate}
\item All the blocks are of split geometry
\item Each split geometry block $B_i$ has a Margulis tube $T_i$ corresponding to curve $v_i$
splitting it into an $S_{(0,3)i}$ and an $S_{(0,n-1)i}$. We call them $A_i, C_i$ for convenience.
Thus the vertical boundary of $T_i$ has two sides. The {\it short}  vertical boundary
 abutting   $A_i$
has length one and the {\it long}  vertical boundary abutting   $C_i$
has length $m_i$.
\item There is a hanging tube $T_{i,i+1}$ (see Section \ref{ht}) denoted $H_i$
corresponding to curve $v_{i,i+1}$  intersecting only $B_i$ and  $B_{i+1}$ 
separating the successive $A_i$ and $A_{i+1}$  so that the shortest path between 
$A_i$ and $A_{i+1}$ necessarily passes through $H_i$.
\item Arrange so that for all $i$,
 any geodesic in the hierarchy supported in a proper subsurface of $S\setminus \{ v_i, 
v_{i+1}, 
v_{i,i+1} \}$ is of bounded length; equivalently all curves other than 
$ \cup_i \{ v_i, 
v_{i+1}, 
v_{i,i+1} \}$ have geodesic realization in $M$ of length at least $\epsilon_0 > 0$.
\end{enumerate}

Once this is done, we see that 
\begin{enumerate}
\item Each split block is of bounded thickness (given by the part of $S$ corresponding to $A_i$).
\item Any almost minimizing geodesic necessarily passes deep into $H_i$ and is therefore thin.
\end{enumerate}

Towards this, it suffices to construct a hierarchy path (see Section \ref{hier} 
in the Appendix) such that conditions 3 and 4 above are
satisfied. We translate this into the language of resolutions of a hierarchy to obtain the
 qualitative properties of the resolution we want.
We require that

\begin{enumerate}
\item The geodesic in the hierarchy on the subsurface $\Sigma_{i, i+1}$ bounded by $\{ v_i, 
v_{i,i+1} \}$ is long and thick, for instance such that $\Sigma_{i, i+1}$ is $S_{0,4}$ and the corresponding
blocks are thick. 
\item The same for the geodesic in the hierarchy on the subsurface $\Sigma_{i+1, i}$ bounded by $\{ v_{i+1}, 
v_{i,i+1} \}$
\end{enumerate}

Towards this we outline the construction in blocks $B_i$
and the split surfaces in them.
 We will have numbers $m_i, l_i, n_i$ corresponding
to block $B_i$ and then determine quite flexible conditions on them to satisfy the above requirements.

\begin{enumerate}
\item The  length of the long side of $T_i$ is $n_i$ (equivalently,  the long side of $T_i$
has $n_i$ Minsky blocks abutting it).   
\item There is a constant $L$ such that the number of $L-$bilipschitz split surfaces having $v_i$ and 
$v_{i-1,i}$ as boundary curves is $l_i$ and successive such split surfaces (corresponding to the
resolution) bound between them a region that is $L-$bilipschitz to the product  $\Sigma_{i, i-1} \times I$.
\item Similarly, the number of $L-$bilipschitz split surfaces having $v_i$ and 
$v_{i,i+1}$ as boundary curves is $m_i$ and successive such split surfaces (corresponding to the
resolution) bound between them a region that is $L-$bilipschitz to the product  $\Sigma_{i+1, i} \times I$.
\item $n_i \geq (l_i + m_i)$. This ensures that there is an $L-$bilipschitz  split surface with only
$v_i$ as its boundary component, situated between the top of the lower hanging tube and the
bottom of the higher hanging tube.
\item There exist $L-$bilipschitz split surfaces having only
$v_{i-1,i}$ as boundary curves (these correspond to the split surfaces abutting on the vertical boundary of
$T_{i-1,i}$ opposite to the ones considered in (2) above). These are also $l_i$ in number and successive
ones cobound a region $L-$bilipschitz to a product.
\item There exist $L-$bilipschitz split surfaces having only
$v_{i,i+1}$ as boundary curves (these correspond to the split surfaces abutting on the vertical boundary of
$T_{i,i+1}$ opposite to the ones considered in (3) above). These are also $m_i$ in number and successive
ones cobound a region $L-$bilipschitz to a product.
\item $l_i, m_i \to \infty$ as $i \to \infty$.
\end{enumerate}

See figure below.

\smallskip

\begin{center}

\includegraphics[height=12cm]{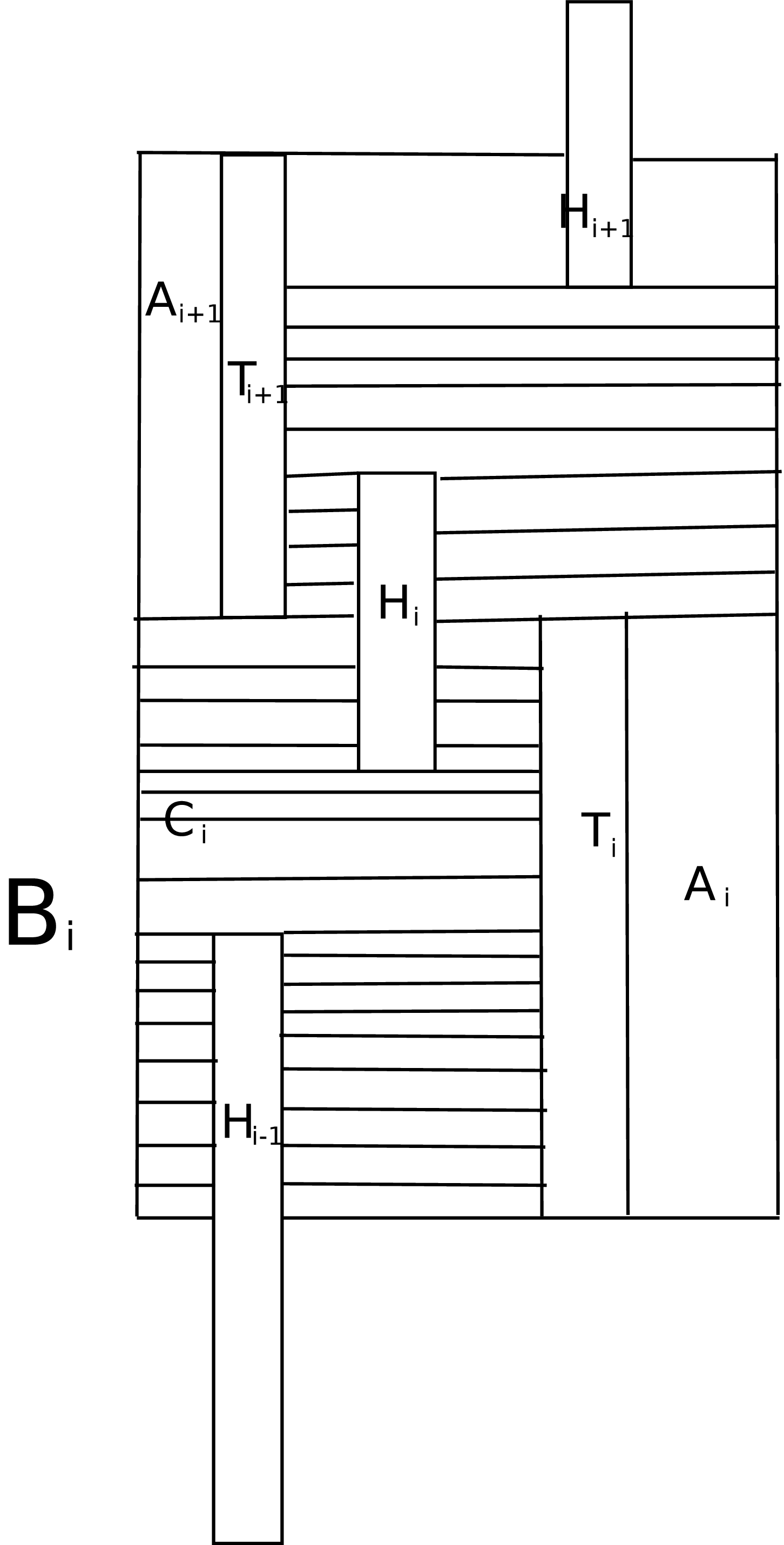}

\smallskip

\end{center}

\smallskip

\noindent {\bf The Estimates:}\\
With the conditions above satisfied, the promised counterexample is a consequence of the following argument. 
The 
distance from  the bottom of $B_i$ to the top of $B_{i+1}$
is  approximately $1+ log (2l_i)$ if one cuts across
the hanging tube. Else, any thick path has length at least $m_{i+1}$. We can make $m_{i+1}$
and $l_i$ comparable, forcing the thin path to be shorter (as it is logarithmic in the length of the shortest
thick path).\\

\noindent {\bf Constructing slices and the hierarchy}\\
It therefore suffices to find a geodesic in the curve complex such that its associated 
 hierarchy path
satisfies  the above conditions. Let us start at the middle of block $B_i$ to see how to build up the hierarchy
(the hierarchy described below is a somewhat more sophisticated version of the well-known `chariot-wheel'
example  \cite{masur-minsky2, minsky-elc1}):

\begin{enumerate}
\item The sequence of split surfaces (or equivalently, slices of the hierarchy)
 give a thick geodesic supported in $C_i$
\item The geodesic stops at a  slice containing $v_i, v_{i,i+1}$. This corresponds to the 
slice through the lower boundary
of the hanging tube $H_i$ in the picture.
\item Next we have two thick geodesics in a pair of component domains  properly contained in $C_i$.
These component domains correspond to the two components of $C_i \setminus v_{i,i+1}$.
\item In the resolution, the  hierarchy geodesics of the previous item
 end at the last slice containing $v_i$ and corresponds to the top
boundary of $B_i$.
\item At this stage $v_i$ is replaced by $v_{i+1}$ and we have two thick geodesics in a pair of component domains  properly contained in $C_{i+1}$.
These component domains correspond to the two components of $C_{i+1} \setminus v_{i,i+1}$.
\item In the resolution, the  hierarchy geodesics of the previous item
 ends with the last slice containing both $v_{i+1}, v_{i,i+1}$ and corresponds to the 
slice through the upper boundary
of $H_i$ in the picture.
\item The next set of slices of the hierarchy
 give a thick geodesic supported in $C_{i+1}$ 
\item We can now go back to item (1), replacing $i$ by $i+1$ and proceeding as above.
\end{enumerate}

It is now possible to construct a geodesic path exhibiting the above behavior by
choosing the geodesics in the links
of $v_i$ and $\{ v_i, v_{i,i+1} \}$ according to the requirements given by Items (1) and (3) above. This completes
the construction of our counterexample.

\section*{Acknowledgments} This work was initiated during a visit of the second author to Universite Paul
Sabatier, Toulouse,  during June 2015. We thank the University for its hospitality.

The first author would like to thank Y. Coud\`ene for very stimulating conversations.
\bibliography{deghor4}
\bibliographystyle{alpha}

\newpage

\section{Appendix: Hierarchies and Split Geometry} \label{splitt}

We recapitulate the essential aspects of hierarchies and split geometry from \cite{masur-minsky2, minsky-elc1, mahan-split}. The definitions here follow \cite{mahan-split}.

\subsection{Hierarchies} \label{hier}
 We fix some notation first:\\
\begin{itemize}
\item $\xi ( S_{g,b} ) =
3g+b$ is the
 {\it complexity} of a compact surface $S=S_{g,b}$ of genus $g$
and $b$ boundary components 
\item For an essential subsurface $Y$ of $S$, $\CC(Y)$ will be its curve complex and $\PP (Y)$ its pants complex.
\item $\gamma_\alpha$ will be a collection of disjoint simple closed curves on $S$ corresponding to a
 simplex $\alpha \in \CC(Y)$
\item  $\alpha,\beta$ in   $\CC(Y)$ 
 fill  an essential subsurface $Y$ of $S$ if all non-trivial non-peripheral curves in $Y$ have  essential
intersection with at least one of $\gamma_\alpha$ or $\gamma_\beta$, where  $\gamma_\alpha$ and $\gamma_\beta$
are chosen to intersect  minimally. 
\item 
Given  $\alpha,\beta$ in $\CC(S)$,
form a regular neighborhood of $\gamma_\alpha\cup\gamma_\beta$, 
and fill in all
disks and one-holed disks to obtain $Y $ which is said to be {\it filled} by $\alpha,\beta$. 
\item For an essential subsurface $X\subset Z$ let $\boundary_Z(X)$ denote the
{\em relative boundary} of $X$ in $Z$, i.e.\ those boundary components
of $X$ that are non-peripheral in $Z$.
\item A   {\bf pants decomposition} of a compact surface $S$, possibly with boundary,
is a disjoint collection of
3-holed spheres $P_1, \cdots , P_n$ embedded in $S$ such that $S \setminus \bigcup_i P_i$ is a disjoint collection of non-peripheral
annuli in $S$, no two of which are  homotopic. 
\item A {\bf tube} in an end $E \subset N$ is a regular $R-$neighborhood
$N(\gamma, R)$ of an unknotted geodesic $\gamma$ in $E$.
\end{itemize}

\begin{definition} {\bf 
Tight Geodesics and Component Domains}\\
{\rm
  Let $Y$ be an essential subsurface in $S$. If $\xi(Y)>4$,  a } {\bf tight} {\rm
  sequence of simplices
  $\{v_i\}_{i\in\II} \subset \CC(Y) $ (where $\II$ is a finite or
  infinite interval in $\Z$) satisfies the following: \\
1) For any vertices $w_i$ of $v_i$ and $w_j$ of $v_j$ where $i\ne
  j$, $d_{\CC(Y)}(w_i,w_j) = |i-j|$,\\
2) For $\{i-1,i,i+1\}\subset \II$, $v_i$ equals $\boundary_Y F(v_{i-1},v_{i+1})$.\\
If $\xi(Y)=4$ then a tight sequence is  the vertex sequence
of a geodesic in $\CC(Y)$.\\
A } tight geodesic {\rm $g$ in $\CC(Y)$
consists of a tight sequence
$v_0, \cdots , v_n$, and two simplices  in $\CC (Y)$, $\I=\I(g)$ and $\T=\T(g)$,
called its  initial and  terminal markings such that $v_0$ (resp. $v_n$) is a sub-simplex of $\I$ (resp. $\T$). The length of $g$ is $n$. 
$v_i$ is called a  simplex of $g$.
$Y$ is called the }  domain or  support {\rm of $g$ and is
denoted as $Y=D(g)$.  $g$ is said to be  supported in $D(g)$}.
\end{definition}

For a surface $W$ with $\xi(W)\ge 4$ and $v$ a simplex of $\CC(W)$
we say that $Y$ is a
{\em component domain of $(W,v)$} if  $Y$ is a component of
$W\setminus \collar(v)$, where $\collar(v)$ is a small tubular neighborhood of the simple closed curves.

If $g$ is a tight geodesic with domain $D(g)$,
we call $Y\subset S$ a {\em component domain of $g$} if
for some simplex $v_j$ of $g$, $Y$ is a component domain of
$(D(g), v_j)$.

\begin{defn} \label{hpa} {\bf Hierarchies}\\ {\rm  A {\bf hierarchy path} in $\PP (S)$ 
joining pants decompositions $P_1$ and $P_2$  is a path $\rho : [0, n] \rightarrow P(S)$ joining $ \rho (0) = P_1$ to $\rho (n) = P_2$ such that \\
1) There is a collection $\{ Y \}$ of essential, non-annular subsurfaces of $S$, called
       component domains for $\rho$, such that for each component domain $Y$ there is a
       connected interval $J_Y \subset [0, n]$ with $\partial Y  \subset \rho (j)$ for each $j \in  J_Y$.\\
2) For a component domain $Y$, there exists a tight geodesic $g_Y $ supported in $Y$ such that for each
       $j \in J_Y$, there is an $\alpha \in g_Y$ with $\alpha \in \rho (j)$. \\ 
A {\bf hierarchy path} in $\PP (S)$  is a sequence $\{ P_n \}_n$ of pants decompositions of $S$ such that for any
$P_i, P_j \in \{ P_n \}_n$, $i \leq j$,  the  finite sequence  $P_i, P_{i+1}, \cdots , P_{j-1}, P_j$ is a hierarchy path
joining pants decompositions $P_i$ and $P_j$. \\
The collection $H$ of tight geodesics $g_Y $ supported in  component domains  $Y$ of $\rho$ will be called the} {\bf hierarchy} {\rm
of tight geodesics associated to $\rho$.} \end{defn}

\begin{definition}{\rm 
A {\bf slice } of a hierarchy $H$ associated to a hierarchy path $\rho$ is a set $\tau$ of pairs
$(h,v)$, where $h\in H$ and $v$ is a simplex of $h$, satisfying
the following properties:
\begin{enumerate}
\item A geodesic $h$ appears in at most one pair in $\tau$.
\item There is a distinguished pair $(h_\tau,v_\tau)$ in $\tau$,
  called the bottom pair of $\tau$. We call $h_\tau$ the bottom geodesic.
\item For every $(k,w)\in \tau$ other than the bottom pair, $D(k) $
  is a component domain of $(D(h),v)$ for some $(h,v)\in\tau$.
\end{enumerate}

A {\bf resolution} of a hierarchy $H$ associated to a hierarchy path $\rho : I \rightarrow \PP (S)$ is a sequence of slices $\tau_i
=\{ (h_{i1}, v_{i1}), (h_{i2}, v_{i2}), \cdots ,  (h_{in_i}, v_{in_i}) \} $ (for $i \in I$, the same indexing set)
such that the set of vertices of the simplices  $\{  v_{i1},  v_{i2}, \cdots ,  v_{in_i} \} $ is the same as the set of the non-peripheral
boundary curves of the pairs of pants in $\rho (i) \in \PP (S)$.}
\end{definition}

\subsection{Split level Surfaces}\label{sls}
Let $E$ be a degenerate end of a hyperbolic 3-manifold $N$.
Let $\TT$ denote a collection of disjoint, uniformly separated tubes in ends of  $N$ 
such that

\begin{enumerate}
\item
 All Margulis tubes in $E$ belong to $\TT$.
\item there exists $\epsilon_0 >0$ such that the injectivity radius $injrad_x(E) > \epsilon_0$ for all $x \in E \setminus  \bigcup_{T \in \TT} Int(T)$.
\end{enumerate}

In \cite{minsky-elc1}, Minsky constructs a
model manifold $M$ bilipschitz homeomorphic to $N$ and equipped with a piecewise Riemannian structure.

Let $(Q, \partial Q)$ be the unique hyperbolic  pair of pants such that each  component
of $\partial Q$ has length one. $Q$ will be called
the {\it standard} pair of pants.
An isometrically embedded copy of $(Q, \partial Q)$ in $(M(0), \partial M(0))$ will be said to be {\it flat}.

\begin{defn} {\rm A {\bf split level surface} associated to a pants decomposition $\{ Q_1, \cdots , Q_n \}$ of 
a compact surface $S$ (possibly with boundary) in $M(0) \subset M$
is an embedding $f : \cup_i (Q_i, \partial Q_i) \rightarrow (M(0), \partial M(0))$ such that \\
1) Each $f (Q_i, \partial Q_i)$ is flat \\
2) $f$ extends to an embedding (also denoted $f$) of $S$ into $M$ such that the interior of each annulus component of
$f(S \setminus \bigcup_i Q_i)$  lies entirely in $F(\bigcup_{T \in \TT} Int(T))$. \\
} \end{defn}

Let
$S_{i}^{s}$ denote the union of the collection of flat pairs of pants
 in the image of the embedding  $S_{i}$.  Note that $S_i \setminus S_{i}^{s}$ consists of annuli properly embedded in Margulis tubes.

The class   of {\it all} topological embeddings from $S$ to $M$ that agree with a split level surface $f$ 
associated to a pants decomposition $\{ Q_1, \cdots , Q_n \}$ on 
$Q_1 \cup \cdots \cup Q_n$ will be denoted by $[f]$. 

We define a partial order $\leq_E$ on the collection of split level surfaces in an end $E$ of $M$ as follows: \\
$f_1 \leq_E f_2$ if there exist $g_i \in [f_i]$, $i=1,2$, such that $g_2(S)$ lies in the unbounded component of $E \setminus g_1(S)$.

A sequence $S_i$ of split level surfaces is said to exit an end $E$ if $i<j$ implies $S_i \leq_E S_j$ and further for all compact subsets $B \subset E$, there exists
$L>0$ such that $S_i \cap B = \emptyset$ for all $i \geq L$.

\begin{definition}\label{thin-def}
  A curve $v$ in $S \subset E$ is {\bf $l$-thin} if the core curve of the Margulis tube $T_v (\subset E \subset N)$ has length less than or equal to $l$. 
A   tube $T\in \TT$  is   $l$-thin  if its core curve    is   $l$-thin.  A   tube $T\in \TT$  is   $l$-thick if it is not    $l$-thin.  \\
A curve $v$ is said to split a pair of split level surfaces $S_i$ and $S_j$ ($i<j$) if $v$ occurs as a boundary curve of
 both $S_i$ and $S_{j}$. A pair of split level surfaces $S_i$ and $S_j$ ($i<j$) is said to be an {\bf $l$-thin pair} if there exists an   $l$-thin curve $v$ 
 splitting both  $S_i$ and $S_{j}$.  \\
\end{definition}

The collection of all  $l$-thin tubes is denoted as $\TT_l$. The union of all  $l$-thick tubes along with $M(0)$  is denoted as $M(l)$.

\begin{defn}
A pair of split level surfaces $S_i$ and $S_j$ ($i<j$) is said to be {\bf $k$-separated} if \\
a) for all $x \in S_i^s$, 
$d_M(x,S_j^s) \geq k$\\
b)  Similarly, for all $x \in S_j^s$, $d_M(x,S_i^s) \geq k$. \end{defn}

\begin{defn} {\bf  Minsky Blocks}\label{mb}
 (Section 8.1 of \cite{minsky-elc1})\\ 
A tight geodesic in $H$ supported in a component domain of complexity $4$ is called a $4$-geodesic
and an edge of a $4$-geodesic in $H$ is called a $4$-edge.

 Given a 4-edge $e$ in $H$,
let $g$ be the $4$-geodesic containing it, and let $D(e)$ be the
domain $D(g)$. Let $e^-$ and $e^+$ denote the initial and
terminal vertices of $e$. Also $\collar v$ denotes a small tubular neighborhood of $v$ in $D(e)$.

To each $e$  a Minsky block $B(e)$ is assigned as as follows:
$$B(e) =  (D(e)\times [-1,1]) \setminus ( {\bf collar}
(e^-)\times[-1,-1/2)\cup   {\bf collar} (e^+)\times(1/2,1]).$$

 The {\em horizontal
boundary} of $B(e)$ is

\begin{center}

$
\boundary_\pm B(e) \equiv (D(e)\setminus\collar(e^\pm)) \times
\{\pm 1\}.
$

\end{center}

The horizontal boundary is  a union of three-holed spheres. The
rest of the boundary is a union of annuli
and is called the vertical boundary. The top (resp. bottom) horizontal boundaries of $B(e)$ are 
$(D(e)\setminus\collar(e^+)) \times \{ 1 \} $  (resp.  $(D(e)\setminus\collar(e^-)) \times \{ - 1 \} $.
\end{defn}

\subsubsection{The Model and the bi-Lipschitz Model Theorem}

\begin{theorem} \cite{minsky-elc1} \cite{minsky-elc2} 
Let $N$ be the convex core of a simply or doubly degenerate hyperbolic 3-manifold minus an open neighborhood of the cusp(s).
Let $S$ be a compact surface, possibly with boundary, such that
$N$ is homeomorphic to $S \times [0, \infty )$ or  $S \times \mathbb{R}$ according as 
$N$ is simply or doubly degenerate. 
 There exist $L \geq 1$, $ \theta, \omega, \epsilon, \epsilon_1 > 0$,
  a collection $\TT$ of $(\theta,\omega)$-thin tubes containing all Margulis tubes in $N$,
a  3-manifold  $M$,  and an $L$-bilipschitz homeomorphism $F: N \rightarrow M$ 
such that the following holds. \\
Let $M(0) = F(N \setminus \bigcup_{T \in \TT} Int(T))$ and let $F(\TT)$ denote the image of the collection $\TT$ under $F$.
Let  $\leq_E$ denote the partial order on the collection of split level surfaces in an end $E$ of $M$. Then there exists a sequence
$S_i$ of split level surfaces associated to pants decompositions $P_i$ exiting $E$ such that  

\begin{enumerate}
\item $S_i \leq_E S_j$ if $i \leq j$.
\item The sequence $\{ P_i \}$ is a hierarchy path in $\PP (S)$.
\item If $P_i \cap P_j = \{ Q_1, \cdots Q_l \}$ then $f_i(Q_k)=f_j(Q_k)$ for $k=1 \cdots l$, where $f_i, f_j$ are the embeddings
defining the split level surfaces $S_i, S_j$ respectively.
\item For all $i$, $P_i \cap P_{i+1} = \{ Q_{i,1}, \cdots Q_{i,l} \}$ consists of a collection of $l $ pairs of pants,
 such that $S \setminus  (Q_{i,1} \cup \cdots \cup Q_{i,l})$
has a single non-annular component of complexity $4$.
Further, there exists a Minsky block $W_i$ and an isometric map $G_i$ of $W_i$ into $M(0)$ such that $f_i(S \setminus
 (Q_{i,1} \cup \cdots \cup Q_{i,l})$ (resp. $f_{i+1}(S \setminus
 (Q_{i,1} \cup \cdots \cup Q_{i,l})$) is contained in the bottom (resp. top) gluing boundary of $W_i$. 
 \item For each flat pair of pants $Q$ in a split level  surface $S_i$ there exists an isometric embedding of $Q \times [-\epsilon, \epsilon]$
into $M(0)$ such that the embedding restricted to $Q \times \{ 0 \}$  agrees with $f_i$ restricted to $Q$.
\item For each $T\in \TT$, there exists a split level surface $S_i$ associated to pants decompositions $P_i$
such that the core curve of $T$ is isotopic to a non-peripheral boundary curve of $P_i$. The boundary $F(\partial T)$ of $F(T)$ with the induced
metric $d_T$ from  $M(0)$ is a
Euclidean torus equipped with a product structure $S^1 \times S^1_v$,
where any circle of the form $S^1 \times \{ t \} \subset S^1 \times S^1_v$
is a round circle of unit length and is called a horizontal circle; and any circle of the form 
 $\{ t \} \times  S^1_v$ is a round circle of length $l_v$ and is called a vertical circle.
\item Let $g$ be a tight geodesic other than the bottom geodesic
 in the hierarchy $H$ associated to the hierarchy path $\{ P_i \}$, let $D(g)$ 
be the support of $g$ and let $v$ be a boundary curve of $D(g)$. Let $T_v$ be the tube in $\TT$ such that
the core curve of $T_v$ is isotopic to $v$. If a vertical circle of $(F(\partial T_v), d_{T_v})$ has length $l_v$ less than
$n\epsilon_1$, then the length of $g$ is less than $n$. \end{enumerate}
\label{bilipmodel} \end{theorem}

\subsubsection{Tori  and Meridinal Coefficients} \label{tori}
Let $T$ be the boundary of a Margulis tube in $M$. The boundary of a
Margulis tube has the structure of  a Euclidean torus and gives a unique
 point $\omega_T$ in the upper half plane, the Teichmuller space of the torus. The real and imaginary
components of $\omega_T$ have a geometric interpretation.

Suppose that the Margulis tube $T$  corresponds to a vertex $v \in \CC (S)$. Let $tw_T$
 be the signed length of the annulus geodesic corresponding to $v$, i.e.\ it counts with sign
the number of Dehn twists about the curve
represented by $v$.  Next, note that by the construction of the Minsky
model, the {\bf vertical} boundary of $T$ consists of two sides - the
{\bf left vertical boundary} and {\bf right vertical boundary}. Each is attached to vertical boundaries of Minsky
blocks. Let the total number of blocks whose vertical boundaries,
 are glued to the vertical boundary
of $T$ be $n_T$. Similarly, let the total number of blocks whose vertical boundaries,
 are glued to the left (resp. right) vertical boundary
of $T$ be $n_{Tl}$ (resp. $n_{Tr}$) so that $n_T = n_{Tl} + n_{Tr}$.

In Section 9 of \cite{minsky-elc1}, Minsky shows:

\begin{theorem} \cite{minsky-elc1}
There exists $C_0 \geq 0$, such that the following holds. \\
$ \omega - (tw_T + i n_T )\leq C_0$
 \label{meridiancoeff}
 \end{theorem}

\subsubsection{Consequences}\cite{mahan-split}
Two consequences of Theorem \ref{bilipmodel} that we shall need are given below.

\begin{lemma} \cite[Lemma 3.6]{mahan-split}
Given $l > 0$ there exists $n \in \natls$ such that the following holds. \\Let $v$ be a vertex in the hierarchy $H$ such that
the length of the core curve of the Margulis tube $T_v$ corresponding to $v$ is greater than $l$. 
Next suppose $(h,v) \in \tau_i$ for some slice $\tau_i$ of the hierarchy $H$ such that $h $
is supported on $Y$, and $D$ is a component of $Y \setminus \collar v$. Also suppose that $h_1 \in H$
such that $D$ is the support of $h_1$. Then the length of $h_1$ is at most $n$.
\label{abut}
\end{lemma}

\begin{lemma}\cite[Lemma 3.7]{mahan-split}
Given $l > 0$ and $n \in \natls$, there exists $L_2 \geq 1$ such that the following holds:\\
Let $S_i, S_j$ ($i<j$) be split level surfaces associated to pants decompositions $P_i, P_j$ such that \\
a) $(j-i) \leq n$ \\
b) $P_i \cap P_j$ is a (possibly empty) pants
decomposition of $S \setminus W$, where $W$ is an essential (possibly disconnected) subsurface of $S$ such that
each component $W_k$ of $W$ has complexity $\xi (W_k) \geq 4$.\\
c)For any $k$ with $i < k < j$, and 
$(g_D, v) \in \tau_k$ for $D \subset W_i$ for some $i$, no curve in $v$ has a geodesic realization in $N$ of length less than $l$. \\
Then there exists an $L_2$-bilipschitz embedding  $G: W \times [-1,1] \rightarrow M$, such that \\
1) $W$ admits a hyperbolic metric given by $W = Q_1 \cup \cdots \cup Q_m$ where each $Q_i$ is a flat pair of pants. \\
2) $ W \times [-1,1]$ is given the product metric. \\
3) $f_i(P_i \setminus P_i \cap P_j) \subset W \times \{ -1 \}$ and $f_j(P_j \setminus P_i \cap P_i) \subset W \times \{ 1 \}$.
\label{thick-lem} \end{lemma}

\subsection{Split surfaces and weak split geometry}

\begin{defn} {\rm An $L$-bi-Lipschitz {\bf split  surface} in $M(l)$ associated to a pants decomposition $\{ Q_1, \cdots , Q_n \}$ of $S$
and a collection $\{ A_1, \cdots , A_m \}$ of complementary annuli (not necessarily all of them) in $S$ 
is an embedding $f : \cup_i Q_i \bigcup  \cup_i A_i \rightarrow M(l)$ such that\\
1) the restriction  $f: \cup_i (Q_i, \partial Q_i) \rightarrow (M(0), \partial M(0))$ is a split level surface \\
2) the restriction $f: A_i \rightarrow M(l)$ is an $L$-bi-Lipschitz embedding.\\
3)  $f$ extends to an embedding (also denoted $f$) of $S$ into $M$ such that the interior of each annulus component of
$f(S \setminus (\cup_i Q_i \bigcup  \cup_i A_i))$  lies entirely in $F(\bigcup_{T \in \TT_l} Int(T))$.}\end{defn}

A split level surface differs from a split surface in that the latter may contain
bi-Lipschitz annuli in addition to flat pairs of pants. We denote split surfaces by  $\Sigma_{i}$.
Let
$\Sigma_{i}^{s}$ denote the union of the collection of flat pairs of pants
and bi-Lipschitz annuli in the image of the split surface (embedding)  $\Sigma_{i}$.

 The next Theorem is one of the technical tools from
\cite{mahan-split}. 

\begin{theorem} \cite[Theorem 4.8]{mahan-split}
Let $N, M, M(0), S, F$ be as in Theorem \ref{bilipmodel} and $E$ an end of $M$. For any $l$ less than the Margulis constant,
let $M(l) = \{ F(x) : {\rm injrad_x} (N) \geq l \}$. Fix a hyperbolic metric on $S$ such that each component of $\partial S$ is 
totally geodesic of length one (this is a normalization condition).
 There exist $ L_1 \geq 1$, $  \epsilon_1 > 0$, $n \in \natls$, 
 and a sequence $\Sigma_i$ of $L_1$-bilipschitz, $  \epsilon_1$-separated split  surfaces exiting the end $E$ of $M$
such that for all $i$, one of the following occurs: 
\begin{enumerate}
\item An $l$-thin curve splits the pair $(\Sigma_i ,\Sigma_{i+1})$, i.e.\ the associated split level surfaces form
an $l$-thin pair. 
\item there exists an $L_1$-bilipschitz embedding  $$G_i: (S\times [0,1], (\partial S)\times [0,1]) \rightarrow (M, \partial M)$$
such that $\Sigma_i^s = G_i (S\times \{ 0\})$ and $\Sigma_{i+1}^s = G_i (S\times \{ 1\})$

\end{enumerate}
Finally, each $l$-thin curve in $S$ splits at most
$n$ split level surfaces in the  sequence $\{ \Sigma_{i} \}$. \label{wsplit}
\end{theorem}

Pairs of split surfaces satisfying Alternative (1) of Theorem \ref{wsplit} will be called an 
{\bf $l$-thin pair} of split surfaces (or simply a thin
pair if $l$ is understood). Similarly, pairs of split surfaces satisfying Alternative (2) of Theorem \ref{wsplit} will be called an {\bf $l$-thick pair}
(or simply a thick
pair) of split surfaces.

\begin{defn} \label{wsplitrmk} 
A model manifold satisfying the following conditions is said
to have {\bf weak split geometry}:\\
\begin{enumerate}
\item A sequence of split surfaces $S^s_i$ exiting the
end(s) of $M$, where $M$ is marked with a homeomorphism to $S \times J$ ($J$ is $\mathbb{R}$ or $[0, \infty )$ according as $M$ is totally or simply degenerate). 
$S_i^s \subset S \times \{ i \}$. \\
\item A collection of Margulis tubes $\mathcal{T}$ in $N$ with image  $F(\mathcal{T})$ in $M$
(under the bilipschitz homeomorphism between $N$ and $M$). We refer to the elements of 
$F(\mathcal{T})$ also as Margulis tubes.\\
\item For each  complementary annulus of $S^s_i$ with core $\sigma$,
  there is a Margulis tube $T\in \mathcal{T}$ whose core is freely homotopic to $\sigma$
   such that $F(T)$ intersects $S^s_i$ at the boundary. (What this roughly
  means is that there is an $F(T)$ that contains the complementary
  annulus.) We say that $F(T)$ splits $S^s_i$.\\
\item There exist constants $\epsilon_0 >0, K_0 >1$ such that
for all $i$, either there exists a Margulis tube splitting both $S^s_i$
  and $S^s_{i+1}$, or else $S_i (=S^s_i)$ and $S_{i+1}  (=S^s_{i+1})$ have injectivity radius bounded
below by $\epsilon_0$ and bound a {\bf thick block} $B_i$, where a thick block is defined to
be a $K_0-$bilipschitz homeomorphic image of $S \times I$. \\
\item $F(T) \cap S^s_i$ is either empty or consists of a pair of
  boundary components of $S^s_i$ that are parallel in $S_i$. \\
\item There is a uniform upper bound $n = n(M)$ on the number of surfaces that
 $F(T)$ splits. 
\end{enumerate}
\end{defn}

\begin{theorem}\cite{mahan-split}
Any  degenerate end of a hyperbolic 3-manifold
 is bi-Lipschitz homeomorphic to a Minsky model
and hence to a model of weak split geometry. \label{minsky-split}
\end{theorem}

\subsubsection{Split Blocks and Hanging Tubes}\label{ht}

\begin{defn} Let $(\Sigma_i^s, \Sigma_{i+1}^s)$ be a thick pair of split surfaces in $ M$. 
The closure of the bounded component of
$M \setminus (\Sigma_i^s \cup \Sigma_{i+1}^s)$ between  $\Sigma_i^s, \Sigma_{i+1}^s$   will be called a thick block.
\end{defn}

Note that a thick block is uniformly bi-Lipschitz to the product $S \times [0,1]$ and that its boundary components are 
$\Sigma_i^s, \Sigma_{i+1}^s$.

\begin{defn} Let $(\Sigma_i^s, \Sigma_{i+1}^s)$ be an $l$-thin pair of split surfaces in $M$
and $F(\TT_i)$ be the collection of $l$-thin Margulis tubes that split both $\Sigma_i^s, \Sigma_{i+1}^s$. The closure of the union of the
bounded components of
$M \setminus ((\Sigma_i^s \cup \Sigma_{i+1}^s)\bigcup_{F(T)\in F(\TT_i)} F(T))$  between  $\Sigma_i^s, \Sigma_{i+1}^s$  will be called a {\bf split block}.

The closure of any bounded component is called a {\bf split component}. 
\end{defn}

Each split component may contain Margulis tubes, which we shall call
{\bf hanging tubes} (see below) that {\em do not} split both $\Sigma_i^s, \Sigma_{i+1}^s$.

Topologically, a {\bf split block} $B^s \subset B = S \times I$
is a topological product $S^s \times I$ for some {\em connected
  $S^s$}. However,  the upper and
lower boundaries of $B^s$ need only be
be split subsurfaces of $S^s$. This is to allow for Margulis tubes
starting (or ending) within the split block. Such tubes would split
one of the horizontal boundaries but not both. We shall call such
tubes {\bf hanging tubes}. See figure below: \\

\begin{center}

\includegraphics[height=4cm]{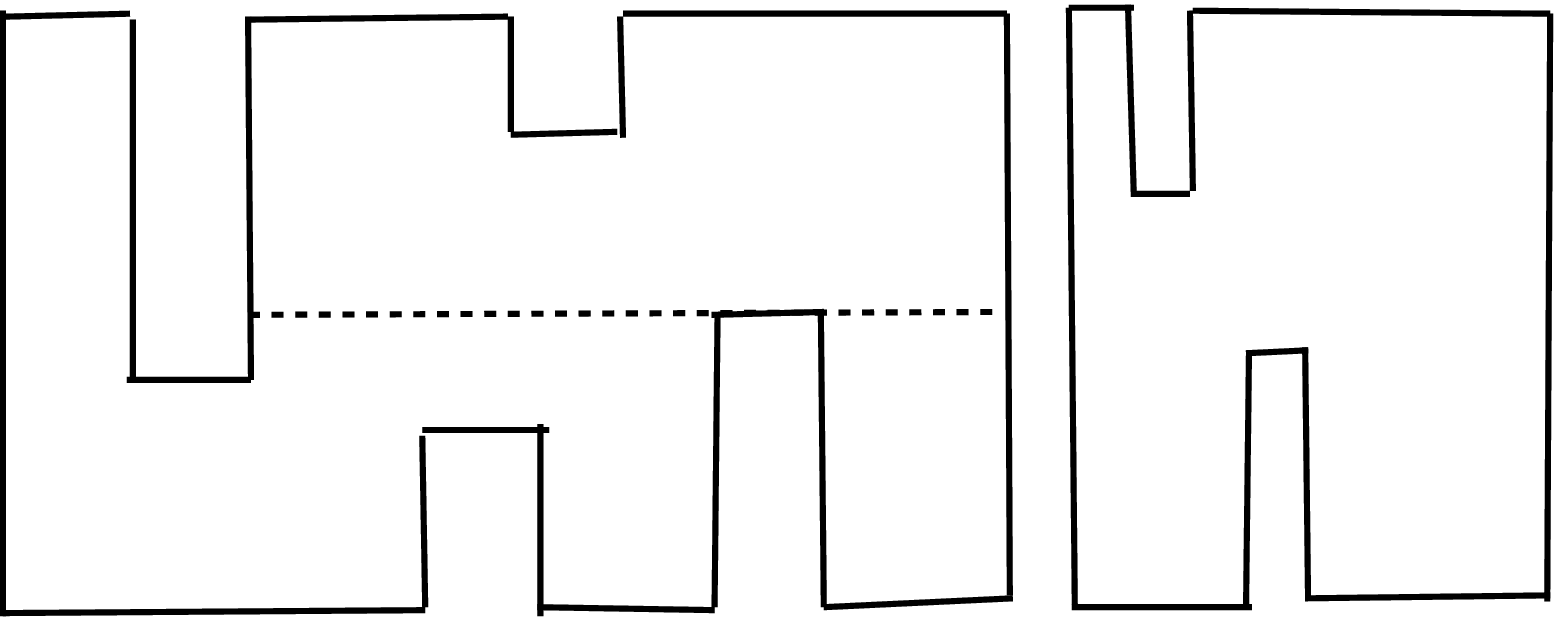}

{\it Split Block with hanging tubes} 

\end{center}

\smallskip
The vertical lengths of {\em hanging tubes} are further required to be
{\em   uniformly bounded
  below by some $\eta_0 > 0$}. Further, each such annulus
has cross section a round circle of length $\epsilon_0$.

\begin{defn}
Hanging tubes intersecting the upper (resp. lower) boundaries of a split block are called 
upper (resp. lower) hanging tubes.
\end{defn}

\section{Erratum}
There are two unfortunate errors in the paper:\\
\begin{enumerate}
\item In the proof of one of the containments in Proposition \ref{multi-almost-minimizing}. The error is in the proof of Claim 3.11.
\item In the appeal to Ledrappier's Theorem 2.12.
\end{enumerate}
We are grateful to James Farre and Yair Minsky for bringing these errors to our notice.

\subsection{Proposition \ref{multi-almost-minimizing}}
The corrected version of  Proposition \ref{multi-almost-minimizing} should state:

\begin{prop}\label{multi-almost-minimizing-corr}
$z \in \Lambda_m$ if $[0,z)$ is almost minimizing.
\end{prop}

The only if direction has a gap in the proof. We suspect that it is not true as stated. The error propagates to Corollary \ref{horo-mult} and 3.13. The rest of the paper is unaffected.

\subsection{Theorems 2.12 and  \ref{am-injrad1}}
Theorem 2.12  which essentially
quotes \cite[Proposition 3]{ledrappier} is  wrong. A corrigendum
to \cite{ledrappier} was brought out by the author
in \cite{ledrappier-corr}. 
Counter examples to Theorem 2.12  have been produced by A. Bellis  \cite{bellis} in the context of surfaces of infinite type.
Theorem \ref{am-injrad1} however only claims the statement for 3-manifolds of finite type. We do not know whether this is true or not. But the  proof
from \cite{ledrappier} that we reference is not complete. Hence the use of  Theorem \ref{am-injrad1} in Theorem \ref{omni} to motivate the subsequent discussion is false. In particular, the motivational content of Theorem \ref{prel} is not valid.
However, since Theorem 2.12 was used mainly to motivate the discussion in the subsequent sections, the content of the subsequent sections remains unaffected.\\

\end{document}